\documentclass[12pt]{extarticle}
\usepackage{amsmath, amsthm, amssymb, hyperref, color}
\usepackage{graphicx}
\usepackage{caption}
\usepackage{subcaption}
\usepackage{mathtools}
\usepackage{enumerate}
\usepackage[all]{xypic}
\usepackage{verbatim}
\usepackage{upgreek}
\usepackage{chemfig, chemnum}
\usepackage{tikz}
\usepackage[lined,commentsnumbered,ruled]{algorithm2e}
\usepackage{caption}
\oddsidemargin \evensidemargin
\theoremstyle{theorem}
\newtheorem{theorem}{Theorem}
\newtheorem{proposition}[theorem]{Proposition}
\newtheorem{lemma}[theorem]{Lemma}
\newtheorem{corollary}[theorem]{Corollary}
\theoremstyle{definition}

\newtheorem{remark}[theorem]{Remark}

\newtheorem{algorithm_thm}[theorem]{Algorithm}
\newtheorem{example}[theorem]{Example}

\newcommand{\PP}{\mathbb{P}}
\newcommand{\RR}{\mathbb{R}}
\newcommand{\QQ}{\mathbb{Q}}
\newcommand{\CC}{\mathbb{C} }
\newcommand{\ZZ}{\mathbb{Z}}
\newcommand{\NN}{\mathbb{N}}
\newcommand{\e}{\epsilon}

\title{\textbf{KP Solitons from Tropical Limits}}

\author{Daniele Agostini, Claudia Fevola, Yelena Mandelshtam and  Bernd Sturmfels}

 \date{ }
 
\begin{document}

 \maketitle

\begin{abstract}
\noindent 
We study solutions to the Kadomtsev-Petviashvili equation 
whose underlying algebraic curves undergo tropical degenerations.
Riemann's theta function becomes a finite 
exponential sum that is supported on a Delaunay polytope.
We introduce the Hirota variety which parametrizes all 
tau functions arising from such a sum.
We compute tau functions from points on the Sato
Grassmannian that represent 
Riemann-Roch~spaces and we present an algorithm that finds a soliton solution from
a rational nodal curve.
\end{abstract}

\section{Introduction}
\label{sec1}

In the interplay between integrable systems and algebraic geometry \cite{AG,
agostini&co, dubrovin, KX, Krichever, Naktrig, Nak},
the study of complex algebraic curves is  connected to the
 {\em Kadomtsev-Petviashvili (KP)~equation} 
\begin{equation}
\label{eq:KP1}
 \frac{\partial}{\partial x}\left( 4 p_t - 6 p p_x - p_{xxx} \right)\, = \, 3 p_{yy}.
 \end{equation}
This is a heavily studied nonlinear partial differential equation
for an unknown function $p(x, y, t)$ that describes the evolution of certain water
waves.

Given a smooth complex algebraic curve, we write $B$ for its Riemann matrix,
normalized to have negative definite real part. One considers the associated Riemann theta function
\begin{equation}
\label{eq:RTFreal}
\theta \,=\, \theta({\bf z}\, |\, B)\,\,\, = \,\,\,
\sum_{{\bf c} \in \mathbb{Z}^g} {\rm exp} \left[ \frac{1}{2} {\bf c}^T B {\bf c} \, + {\bf c}^T {\bf z} \right].
\end{equation}
Krichever \cite{Krichever} constructed $g$-phase solutions  to the KP equation as follows. Let $\tau(x,y,t)$ be obtained from (\ref{eq:RTFreal}) by setting
${\bf z} =  {\bf u} x + {\bf v} y + {\bf w} t$.
Here, ${\bf u} = (u_1,\ldots,u_g)$,
${\bf v} = (v_1,\ldots,v_g)$,
${\bf w} = (w_1,\ldots,w_g)$ are coordinates on
the weighted projective space $\mathbb{WP}^{3g-1}$ that is defined~by
\begin{equation}
\label{eq:grading}
 {\rm deg}(u_i) = 1 \,, \,\,\,
{\rm deg}(v_i) = 2 \, , \,\,\,
{\rm deg}(w_i) = 3 \quad {\rm for} \,\,\, i = 1,2,\ldots, g. 
\end{equation}
We require that the trivariate tau function $\tau(x,y,t)$ satisfies Hirota's differential equation
 \begin{equation}\label{eq:hirotaF}
\tau \tau_{xxxx} \,-\,4\tau_{xxx}\tau_x \,+\, 3\tau_{xx}^2
\,\,\,+\,\,\,4\tau_{x}\tau_t\,-\, 4 \tau \tau_{xt}\,+\,3\tau \tau_{yy}-3\tau_y^2
\,\,\,= \,\,\,0.
\end{equation}
Under this hypothesis, the following function satisfies (\ref{eq:KP1}),
and we call it the {\em KP solution}:
\begin{equation}
\label{eq:u_tau}
 p(x,y,t) \,\,=\,\, 2\frac{\partial^2}{\partial x^2} \log \tau(x,y,t).
\end{equation}
The {\em Dubrovin threefold} of~\cite{agostini&co} comprises all
points $({\bf u},{\bf v},{\bf w})$ in $\mathbb{WP}^{3g-1}$
for which  (\ref{eq:hirotaF}) holds.

In the present paper we study these objects when
the smooth curve is defined over~a non-archimedean field  $\mathbb{K}$ such as 
$\QQ(\epsilon)$ or the
Puiseux series $\CC \{ \! \{ \epsilon \} \! \}$.
The Riemann matrix~$B_\epsilon$ depends analytically on the parameter $\epsilon$,
and hence so do the tau function and  KP~solution. For
$\epsilon \rightarrow 0$, the function $p(x,y,t)$ becomes
a soliton solution of (\ref{eq:KP1}). Our aim is to compute these degenerations
and resulting KP solitons  \cite{AG, Kodamabook} explicitly using computer~algebra.

In the tropical limit, the infinite sum over $\ZZ^g$ in the theta function becomes a finite sum
\begin{equation}
\label{eq:thetafunction}
 \theta_\mathcal{C} ({\bf z}) \,\, = \,\,
a_1 \,{\rm exp}[ \,{\bf c}_1^T  {\bf z} \,] \, + \,
a_2 \,{\rm exp}[ \,{\bf c}_2^T {\bf z} \, ] \, + \,\cdots \,+\,
a_m \,{\rm exp}[ \,{\bf c}_m^T {\bf z} \,] ,
\end{equation}
where $\mathcal{C} = \{{\bf c}_1,{\bf c}_2,
\ldots,{\bf c}_m\}$ is a certain subset of the integer lattice $\ZZ^g$.
Each lattice point ${\bf c}_i = (c_{i1},\ldots,c_{ig})$ specifies
a linear form
${\bf c}_i ^T {\bf z} = \sum_{j=1}^g c_{ij} z_j$, just like in
(\ref{eq:RTFreal}).
The coefficients ${\bf a} = (a_1,a_2,\ldots,a_m)$ 
are unknowns that serve  as coordinates on the algebraic torus $(\mathbb{K}^*)^m$.

\begin{example}[$g=2$] \label{ex:ClaudiaYelena}
Consider a genus two curve $y^2 = f_i(x)$ where
$f_i(x)$ is a polynomial of degree six with coefficients in $ \QQ(\epsilon)$.
Here are two instances corresponding 
to Figures \ref{trees} and~\ref{graphs}:
\begin{equation}
\label{eq:twocurves} \begin{matrix}
f_1(x) & = & (x-1)(x-2\e)(x-3\e^2)(x-4\e^3)(x-5\e^4)(x - 6\e^5), \\
f_2(x) & = & (x-1)(x-1-\e)(x-2)(x-2-\e)(x-3)(x-3-\e).
\end{matrix}
\end{equation}
Note that $f_2$ is an example for \cite[\S 7]{Nak}.
For any fixed $\epsilon \in \CC^*$, we can compute
the Dubrovin threefold in $\mathbb{WP}^5$, using 
\cite[Theorem 3.7]{agostini&co}, and derive
KP solutions from its points. The difficulty is
to maintain $\epsilon$ as a parameter 
and to understand what happens for $\epsilon \rightarrow 0$.

One configuration in $\ZZ^2$ that arises here is the square
$\,\mathcal{C} = \{ (0,0), (1,0),(0,1),(1,1) \} $. In order for 
the associated theta function $\,\theta_\mathcal{C}  =
 a_{00} + a_{10}\, {\rm exp}[z_1] + 
a_{01} \,{\rm exp}[z_2] + a_{11}\, {\rm exp}[z_1+z_2]\, $
to yield a KP solution, 
the following three polynomial identities are necessary and sufficient:
$$
\begin{matrix}
    u_1^4 - 4 u_1 w_1 + 3 v_1^2   \,=\, 0,&
        \quad ((u_1{+}u_2)^4 - 4 (u_1{+}u_2) (w_1{+}w_2)  
            + 3 (v_1{+}v_2)^2 )  \,a_{00} a_{11}  \\  
  u_2^4 - 4 u_2 w_2 + 3 v_2^2    \,=\,  0, &  \qquad \,\, \,+\,
      ((u_1{-}u_2)^4 - 4 (u_1{-}u_2) (w_1{-}w_2) + 3 (v_1{-}v_2)^2 )\,  a_{01} a_{10} \,=\, 0.
\end{matrix}              $$
If these conditions hold then $p(x,y,t)$ can be
written as a $(2,4)$-soliton by
\cite[\S 2.5]{Kodamabook}.
\end{example}

This article is organized as follows.
In Section \ref{sec2} we review the derivation of tropical Riemann matrices.
Theorem~\ref{degthetathm} characterizes
degenerations of theta functions
from algebraic curves over $\mathbb{K}$.
 Proposition \ref{prop:delaunaypolytopes}
shows that $\mathcal{C} $ is the vertex set of
a Delaunay polytope in $\ZZ^g$.
In Section \ref{sec3} we study the Hirota variety $\mathcal{H}_\mathcal{C}$, which lives in
$ (\mathbb{K}^*)^m \times \mathbb{WP}^{3g-1}$.
A point $({\bf a},({\bf u}, {\bf v},{\bf w}))$ lies on the Hirota variety if and only if
(\ref{eq:hirotaF}) holds for (\ref{eq:thetafunction}).
We saw $\mathcal{H}_\mathcal{C}$ for $g=2$ and $\mathcal{C} = \{0,1\}^2$ in Example~\ref{ex:ClaudiaYelena}.
Theorem~\ref{thm:khovanskii}
 characterizes  the Hirota variety of the $g$-simplex.

A key idea in this paper is to never compute
a Riemann matrix or the theta function~(\ref{eq:RTFreal}).
Instead we follow the approach in \cite{KX,Naktrig,Nak}
that rests on the  Sato Grassmannian (Theorem \ref{thm:sgm}).  This setting
is entirely algebraic and hence amenable to symbolic computation over $\mathbb{K}$.
Section \ref{sec4} explains the computation of
tau functions from points on the Sato Grassmannian.

In Section \ref{sec5} we start from an algebraic curve $X$ over $\mathbb{K}$. 
Certain Riemann-Roch spaces on $X$
are encoded as points on the Sato Grassmannian.
Following \cite{Nak}, we present an algorithm and its 
\verb|Maple| implementation for 
computing these points and the resulting tau functions, for $X$  hyperelliptic.
Proposition~\ref{prop:cond1and2}
addresses the case when $X$ is reducible.
This is followed up in Section \ref{sec6},
where we present  Algorithm  \ref{alg:sec6}  for 
KP solitons from nodal rational curves.

\section{Tropical Curves and Delaunay Polytopes}
\label{sec2}

We work over a field $\mathbb{K}$ of characteristic zero with a non-archimedean valuation.
Let $X$ be a {\em Mumford curve} of genus $g$, that is,
$X$ is a smooth curve over $\mathbb{K}$ whose Berkovich analytification is a graph 
with $g$ cycles. This metric graph is the tropicalization
${\rm Trop}(X)$ of a faithful embedding of $X$. 
In spite of the recent advances in \cite{jell}, computing
${\rm Trop}(X)$ from $X$ remains a nontrivial task. All our
examples were derived with
methods described in~\cite{brandt}.

If the curve $X$ is hyperelliptic, given
by an equation $y^2 = f(x)$, then
${\rm Trop}(X)$ is a metric graph with a harmonic two-to-one map onto
  the phylogenetic tree encoding the roots of~$f(x)$. The combinatorics
  of harmonic maps is subtle.   We refer
   \cite[Definition 2.2]{bolognese} for an explanation.

\begin{example}[$g=2$] \label{ex:g2} Let $f(x)$ be of degree six. The six roots
determine a subtree with six leaves in the Berkovich line.
The edge lengths are invariants
of the semistable model \cite{brandt} over
the valuation ring of $\mathbb{K}$.
There are two combinatorial types of trivalent trees
with six leaves, the {\em caterpillar}
and the {\em snowflake}. These are realized by the two polynomials in
Example~\ref{ex:ClaudiaYelena}.

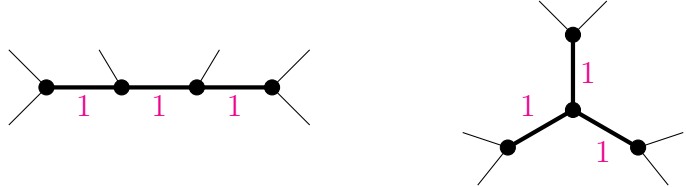
\begin{figure}[ht]
	\centering
	\begin{tikzpicture}
	\draw[fill] (1,0) circle [radius=0.1];
	\draw[fill] (2,0) circle [radius=0.1];
	\draw[fill] (3,0) circle [radius=0.1];
	\draw[fill] (4,0) circle [radius=0.1];
	\draw[ultra thick] (1,0)--(4,0);
	\draw[thin] (.5, 0.5)--(1, 0);
	\draw[thin] (.5, -0.5)--(1, 0);
	\draw[thin] (1.7, 0.5)--(2, 0);
	\draw[thin] (3.3, .5)--(3, 0);
	\draw[thin] (4.5, .5)--(4, 0);
	\draw[thin] (4.5, -.5)--(4, 0);
	\node[color=magenta] at (1.5,-.25) {1};
	\node[color=magenta] at (2.5,-.25) {1};
	\node[color=magenta] at (3.5,-.25) {1};
	
	\draw[fill] (8,-.3) circle [radius=0.1];
	\draw[fill] (8,.7) circle [radius=0.1];
	\draw[fill]({8-sqrt(3)/2},{-.3-1/2}) circle [radius=0.1];
	\draw[fill]({8+sqrt(3)/2},{-.3-1/2}) circle [radius=0.1];
	\draw[ultra thick]({8-sqrt(3)/2},{-.3-1/2})--(8,-.3)--(8,.7);
	\draw[ultra thick]({8+sqrt(3)/2},{-.3-1/2})--(8,-.3);
	\draw[thin]({8-sqrt(3)/2},{-.3-1/2})--({8-sqrt(3)/2 - .6 },{-.3-1/2+.2});
	\draw[thin]({8-sqrt(3)/2},{-.3-1/2})--({8-sqrt(3)/2 - .4 },{-.3-1/2-.5});
	\draw[thin](8,.7) -- (8.45, 1.15);
	\draw[thin](8,.7) -- (7.55, 1.15);
	\draw[thin] ({8+sqrt(3)/2},{-.3-1/2}) -- ({8+sqrt(3)/2 +.6},{-.3-1/2 +.2});
	\draw[thin] ({8+sqrt(3)/2},{-.3-1/2}) -- ({8+sqrt(3)/2 +.4},{-.3-1/2 -.5});
	\node[color=magenta] at (8.2,.2) {1};
	\node[color=magenta] at (8.4,-.85) {1};
	\node[color=magenta] at (7.4,-.25) {1};
	\end{tikzpicture}
	\caption{The metric trees defined by the polynomials $f_1$ (left) and $f_2$ (right) in (\ref{eq:twocurves})}
	\label{trees}
\end{figure}

%\begin{figure}[ht]
%	\centering
%	\includegraphics[scale=0.35]{trees.png} 
%	\caption{The metric trees defined by the polynomials $f_1$ (left) and $f_2$ (right) in (\ref{eq:twocurves})}
%	\label{trees}
%\end{figure}

Each trivalent metric tree with $2g+2$ leaves has a unique hyperelliptic covering
by a metric graph of genus $g$.
This is the content of \cite[Lemma 2.4]{bolognese}.
 The edge lengths of the genus $g$ graph are obtained from those
of the tree by stretching or shrinking with a factor of $2$.
   Figure \ref{graphs} shows the graphs that give a two-to-one map to 
   the trees in Figure \ref{trees}.

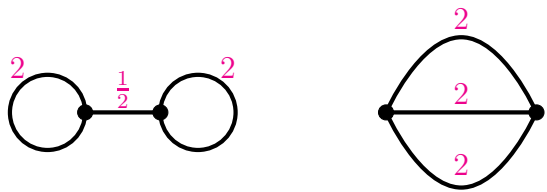
\begin{figure}[ht]
\centering
\begin{tikzpicture}
\draw[fill] (1,0) circle [radius=0.1];
\draw[fill] (2,0) circle [radius=0.1];
\draw[ultra thick] (0.5, 0) circle [radius = 0.5];
\draw[ultra thick] (2.5, 0) circle [radius = 0.5];
\draw[ultra thick] (1,0)--(2,0);
\node[color=magenta] at (.1,.6) {2};
\node[color=magenta] at (2.9,.6) {2};
\node[color=magenta] at (1.5,.3) {$\frac{1}{2}$};

\draw[fill] (5,0) circle [radius=0.1];
\draw[fill] (7,0) circle [radius=0.1];
\draw[ultra thick] (5,0) to (7,0);
\draw [ultra thick] plot [smooth, tension=1] coordinates { (5,0) (6,1) (7,0)};
\draw [ultra thick] plot [smooth, tension=1] coordinates { (5,0) (6,-1) (7,0)};
\node[color=magenta] at (6,.25) {2};
\node[color=magenta] at (6,1.25) {2};
\node[color=magenta] at (6,-.7) {2};

\end{tikzpicture}

	\caption{The metric graphs  ${\rm Trop}(X)$ for the 
		curves $X$ in Example~\ref{ex:ClaudiaYelena} and Figure \ref{trees}}
	\label{graphs}
\end{figure}

%\begin{figure}[ht]
%	\centering
%	\includegraphics[scale=0.45]{tropcurves.png} 
%	\caption{The metric graphs  ${\rm Trop}(X)$ for the 
%curves $X$ in Example~\ref{ex:ClaudiaYelena} and Figure \ref{trees}}
%	\label{graphs}
%\end{figure}

\end{example}

From the tropical curve ${\rm Trop}(X)$ we can read off
the tropical Riemann matrix $Q$. This is a 
positive definite real symmetric $g \times g$ matrix.
Fix a basis of cycles in ${\rm Trop}(X)$ and write them
as the $g$ rows of a matrix $\Lambda$ whose columns are
labeled by the edges. Let $\Delta$ be the diagonal matrix
whose entries are the edge lengths of ${\rm Trop}(X)$. Then we have
$Q = \Lambda \Delta \Lambda^T$.

The genus two graphs in Figure \ref{graphs} have three edges. Their Riemann matrices are
$$
Q_1   =  \begin{small}
\begin{bmatrix}
1 & \!\! 0& \!\! 0 \\
0 & \!\! 0 & \!\! 1 
\end{bmatrix} \!
\begin{bmatrix}
2 & 0 & 0\\
0 & \frac{1}{2} & 0\\
0 & 0 & 2
\end{bmatrix} \!
\begin{bmatrix}
1 & \!\! 0\\
0 & \!\! 0\\
0 & \!\! 1
\end{bmatrix}  =  
\begin{bmatrix}
2 & 0\\
0 & 2
\end{bmatrix} \end{small} ,\,\,
Q_2  = \begin{small}
\begin{bmatrix}
1 &\!\!\!\! -1& \!\! 0\\
 0 & \!1 & \!\!\!\! -1
\end{bmatrix} \!
\begin{bmatrix}
2 & 0 & 0\\
0 & 2 & 0\\
0 & 0 & 2
\end{bmatrix}\!
\begin{bmatrix}
\, 1 & \!0\\
\! -1 & \!1\\
\,0 & \!\!\!\! -1
\end{bmatrix}  =  
\begin{bmatrix}
\,\, 4 & \!\!\!\! -2\\
-2 & 4
\end{bmatrix}.\end{small}
$$

We now consider the degeneration of our curve $X$ over 
$\mathbb{K} = \CC\{ \! \{ \epsilon \} \! \}$
to its tropical limit. The Riemann matrix
can be written in the form
$B_\epsilon = -\frac{1}{\e} Q + R(\e)$,
where $R(\e)$ is a symmetric $g \times g$ matrix
whose entries are complex analytic functions in $\e$
that converge as $\e \rightarrow 0$.

Fix a point $\mathbf{a} \in \RR^g$.
%In the setting of Section \ref{sec5} this should
% correspond to the choice of a divisor on the curve.
Replacing ${\bf z}$ by ${\bf z} + \frac{1}{\e} Q {\bf a}$ in  the Riemann theta function
 (\ref{eq:RTFreal}), we obtain
\begin{equation}
\label{eq:thetaepsilon}
\theta({\bf z} + \frac{1}{\e}Q {\bf a}\, |\, B_\epsilon)\,\,\, = \,\,\,
\sum_{{\bf c} \in \mathbb{Z}^g} {\rm exp} \left[ -\frac{1}{2 \e} {\bf c}^T Q {\bf c} 
+ \frac{1}{\e} {\bf c}^T Q {\bf a} \right] \cdot {\rm exp} \left[
  \frac{1}{2} {\bf c}^T R(\e) {\bf c}  + {\bf c}^T {\bf z} \right].
\end{equation}
This expression converges for $\e \rightarrow 0$ provided
${\bf c}^T Q {\bf c} - 2 {\bf c}^T Q {\bf a} \geq 0$ for all ${\bf c} \in \mathbb{Z}^g$. Equivalently,
\begin{equation}
\label{eq:voronoicell}
\quad \mathbf{a}^T Q \mathbf{a} \,\leq \,(\mathbf{a}-{\bf c})^T Q(\mathbf{a}-{\bf c})
\quad \hbox{for all ${\bf c} \in \mathbb{Z}^g$.}
\end{equation}
This means that the distance from $\mathbf{a}$ to the origin,
in the metric given by $Q$,  is at most the distance to any other lattice point ${\bf c} \in \ZZ^g$.
In other words,  (\ref{eq:voronoicell}) means that
 $\mathbf{a}$ belongs to the {\em Voronoi cell} for $Q$.
 Under this hypothesis, we now consider the associated 
   \emph{Delaunay set}
 \begin{equation}
 \label{eq:delaunayset}
 \mathcal{D}_{\mathbf{a}, Q} \,\, = \,\, \bigl\{\, {\bf c} \in \ZZ^g \,: \,
 \mathbf{a}^T Q \mathbf{a} \,=\,
(\mathbf{a}-{\bf c})^T Q (\mathbf{a}- {\bf c}) \bigr\}. 
\end{equation}
This is the set of vertices of a polytope in the Delaunay subdivision of $\ZZ^g$  induced by $Q$.
If ${\bf a}$ is a vertex of the Voronoi cell then the Delaunay polytope 
${\rm conv}( \mathcal{D}_{\mathbf{a}, Q})$
 is $g$-dimensional.

As in \cite[\S 4]{struwe}, we observe that
${\rm exp} \left[- \frac{1}{2 \e} {\bf c}^T Q {\bf c} 
+ \frac{1}{\e} {\bf c}^T Q {\bf a} \right] $ converges
to $1$ for ${\bf c} \in \mathcal{D}_{\mathbf{a}, Q}$ and to $0$ for
 ${\bf c} \in \ZZ^g \backslash \mathcal{D}_{\mathbf{a}, Q}$.
 We have thus derived the following generalization of
 \cite[Theorem~4]{struwe}:

\begin{theorem}\label{degthetathm}
Fix $\mathbf{a}$ in the Voronoi cell  
of the tropical Riemann matrix $Q$. For $\epsilon \rightarrow 0$,~the series (\ref{eq:thetaepsilon})
converges to  a theta function
(\ref{eq:thetafunction}) supported on the Delaunay set 
  $\,\mathcal{C} = \mathcal{D}_{{\bf a},Q}$, namely
  \begin{equation} \qquad
    \theta_\mathcal{C}(\mathbf{x}) \,\,
    = \,\,\sum_{{\bf c} \in \mathcal{C}} a_{\bf c} \,
        {\rm exp} \! \left[ {\bf c}^T {\bf z} \right],
\quad {\rm where} \quad a_{\bf c} =      {\rm exp}
        \left[\frac{1}{2} {\bf c}^T R(0) {\bf c}   \right].  \end{equation}
\end{theorem}

The Delaunay polytope ${\rm conv}(\mathcal{C})$
can have any dimension between $0$ and $g$, depending on the
location of ${\bf a}$ in the Voronoi cell (\ref{eq:voronoicell}).
If $\mathbf{a}$ lies in the interior then $\mathcal{C} = \{ {\bf 0} \}$ is just the origin.
We are most interested in the case when ${\bf a}$ is a vertex of the Voronoi cell,
and we now assume this to be the case. This ensures that
$\mathcal{C}$ is the vertex set of a $g$-dimensional Delaunay polytope.
For fixed $g$, there is only a finite list of  Delaunay polytopes,
up to lattice isomorphism.
Thanks to \cite{dutour} and its references,
 that list is known for $g \leq 6$. However,
 not every Delaunay polytope arises from a
 curve $X$ and its tropical Riemann matrix $Q = \Lambda \Delta \Lambda^T$.
 To illustrate these points, we present the list of all relevant Delaunay polytopes for $g \leq 4$.

\begin{proposition}\label{prop:delaunaypolytopes}
The complete  list of Delaunay polytopes $\mathcal{C}$ arising from metric graphs for $g \leq 4$ is as follows.
For $g=2$ there are two types:  triangle and square.
For $g=3$ there are five types: tetrahedron, square-based pyramid,
octahedron, triangular prism and cube. For $g=4$ there are
$17$ types. These $\mathcal{C}$ have between $5$ and $16$ vertices. They are
listed in  Table~\ref{table:erdahl}.
\end{proposition}

\begin{proof}
For any edge $e$ of the graph, let 
 $\lambda_e \in \ZZ^g$ be the  associated
column of $\Lambda$.
 The Voronoi cell is a zonotope, obtained
by summing line segments parallel to $\lambda_e$ for all $e$. It has ${\bf a}$ as a vertex.
 After reorienting edges,
in the corresponding expression of ${\bf a}$ as a linear combination of the vectors $\lambda_e$,
all coefficients are positive. This means that the Delaunay polytope equals
\begin{equation}
\label{eq:betweenpoly}
 {\rm conv}(\mathcal{C}) \quad = \quad \bigl\{\,
{\bf c} \in \RR^g \,: \, 0 \leq \lambda_e^T {\bf c} \leq 1 \,\,\hbox{for all edges} \,\, e \,\bigr\}. 
\end{equation}
Our task is to classify the polytopes (\ref{eq:betweenpoly})
for all graphs of genus $g$ and all their orientations.
For $g=2$ this is easy, and for $g=3$ it was done in
\cite[Theorem 4]{struwe}. We see from \cite[Figure 2]{struwe}
 that every Delaunay polytope can be realized by a curve over $\mathbb{K}$.
 For $g=4$ we started from the classification of $19$ Delaunay polytopes in \cite[Theorem 6.2]{erdahl},
 labeled $1,2,\ldots,16$ in \cite[Table V]{erdahl} and labeled A,B,C in  \cite[Table VI]{erdahl}.
Two types do not arise from graphs, namely the
pyramid over the octahedron (\#B) and the cross polytope (\#C).
The other $17$ Delaunay polytopes all arise from graphs.
They are listed in Table~\ref{table:erdahl}. The second row gives the number of vertices.
The third row gives the number of facets.
These two numbers uniquely identify the isomorphism type of $\mathcal{C}$.
The last row indicates which graphs give rise to that Delaunay polytope.
We refer to the $16$ graphs of genus $4$
by the labeling used in \cite[Table 1]{CKS}.
Table~\ref{table:erdahl} was constructed by a direct computation.
It establishes the $g=4$ case in Proposition \ref{prop:delaunaypolytopes}.
\end{proof}

\begin{table}[h]
\setcounter{MaxMatrixCols}{20}
\vspace{-0.15in}
$$
\begin{matrix}
\hbox{polytopes}& 
1 & 2 & 3 & 4 & 5 & 6 & 7 & 8 & 9 & 10 & 11 & 12 & 13 & 14 & 15 & 16 & {\rm A}  \\
\hbox{vertices}  &  5 & 6 & 7 & 7 & 8 & 8 & 8 & 9 & 9 & 9 & 10 & 10 & 10 & 12 & 12 & 16 & 6 \\
\hbox{facets}    &  5 & 6 & 6 & 8 & 7 & 9 & 6 & 7 & 9 & 6 & 7 & 12  & 10   & 7    & 10     & 8 & 9 \\
\hbox{graphs} &
\!  {1,2,3,4 \atop 5,7,10,13} & 
{1,3,4 \atop 5,6,9} & 
 3,\! 7,\!10  & 
4 &
 7 & 
5 & 
 8,\!11,\!15 &
 6 &
 10 &
 12 &
11 &
 9 & 
 13 &
 12 &
 15 &
 16 &
  2 
\end{matrix}  \vspace{-0.2in}
$$
\caption{The $17$ Delaunay polytopes
that arise from the $16$ graphs of genus $4$.
Polytopes are labeled as in \cite[Tables V and VI]{erdahl}
and graphs are labeled as in \cite[Table 1]{CKS}.
For instance, the complete bipartite graph $K_{3,3}$ is \#2, and it has
two Delaunay polytopes, namely the simplex (\#1) and  the cyclic $4$-polytope with $6$ vertices (\#A).
The polytope \#3 has $7$ vertices and $6$ facets. It is the
 pyramid over the triangular prism, and it arises from three graphs (\#3,7,10).}
	\label{table:erdahl}
\end{table}

\section{Hirota Varieties}
\label{sec3}

Starting from the theta function  of the configuration $\mathcal{C}$
in (\ref{eq:thetafunction}), we consider the tau function
$$ \tau(x,y,t) \,\, = \,\, \theta_\mathcal{C}( {\bf u} x + {\bf v} y + {\bf w} t ) 
\,\, = \,\, \sum_{i=1}^m a_i \,{\rm exp} \biggl[
\bigr(\sum_{j=1}^g c_{ij} u_j \bigr)\, x\,\,+\,
\bigr(\sum_{j=1}^g c_{ij} v_j \bigr) \,y\,\,+\,
\bigr(\sum_{j=1}^g c_{ij} w_j \bigr) \,t
  \biggr].
$$ 
The {\em Hirota variety} $\mathcal{H}_\mathcal{C}$  consists of all
 points $\bigl({\bf a}, ({\bf u},{\bf v}, {\bf w})\bigr)$ in the parameter space
$(\mathbb{K}^*)^m \times \mathbb{WP}^{3g-1}$ such that
$\tau(x,y,t)$ satisfies Hirota's differential equation  (\ref{eq:hirotaF}).
Thus  $\mathcal{H}_\mathcal{C}$  is an analogue to the Dubrovin threefold \cite{agostini&co}
for the classical Riemann theta function of a smooth curve.
 
We recall from \cite[equation (2.25)]{Kodamabook} that (\ref{eq:hirotaF})
can be written via the Hirota differential operators as $P(D_x,D_y,D_t) \tau \bullet \tau = 0$,
for the special polynomial $\, P(x,y,t) \, =\, x^4 - 4 xt + 3 y^2$.
For any fixed index $j$, the equation $P(u_j,v_j,w_j) = 0$
defines a  curve in the weighted projective plane $\mathbb{WP}^2$.
More generally, for any two indices $k, \ell $ in $\{1,\ldots,m\}$, we 
consider the hypersurface 
in $\mathbb{WP}^{3g-1}$ defined by
$$
P_{k \ell}({\bf u},{\bf v},{\bf w}) \,\,\, := \,\,\,
P \bigl( \,({\bf c}_k - {\bf c}_\ell) \cdot {\bf u},
\,({\bf c}_k - {\bf c}_\ell) \cdot {\bf v},
\,({\bf c}_k - {\bf c}_\ell) \cdot {\bf w} \bigr). $$
This expression is unchanged if we switch $k$ and $\ell$.
The defining equations of the Hirota variety $\mathcal{H}_\mathcal{C}$ can be obtained from the following lemma,
which is proved by direct computation.

\begin{lemma}
The result of applying the differential operator (\ref{eq:hirotaF}) to the function $\tau(x,y,t)$ equals
\begin{equation}
\label{eq:klsum}
 \sum_{1 \leq k < \ell \leq m} a_k a_\ell \,P_{k \ell}({\bf u},{\bf v},{\bf w})\,
{\rm exp} \bigl[\,
(({\bf c}_k{+}{\bf c}_\ell) \cdot {\bf u}) \, x
\,+  (({\bf c}_k{+}{\bf c}_\ell) \cdot {\bf v}) \,y
\,+  (({\bf c}_k{+}{\bf c}_\ell) \cdot {\bf w}) \,t\,
\bigr].
\end{equation}
\end{lemma}

The polynomials defining the Hirota variety of $\mathcal{C}$
are the coefficients we obtain by 
writing (\ref{eq:klsum})  as a linear combination
of distinct exponentials. These correspond to points in the set
$$ \mathcal{C}^{[2]} \,\, = \,\,
\bigl\{ \,{\bf c}_k + {\bf c}_\ell \,\,: \,\, 1 \leq k < \ell \leq  m \,\bigr\}
\quad \subset \,\,\, \ZZ^g. $$
A point ${\bf d}$ in $\mathcal{C}^{[2]}$ is {\em uniquely attained}
if there exists precisely one index pair $(k,\ell)$ such~that
${\bf c}_k + {\bf c}_\ell = {\bf d}$.
In that case, $(k,\ell)$ is a {\em unique pair}, and this contributes
the quartic  $P_{k\ell}({\bf u},{\bf v},{\bf w}) $ to our defining equations.
If ${\bf d} \in \mathcal{C}^{[2]}$ is
not uniquely attained, then the coefficient we seek is 
\begin{equation}
\label{eq:nonunique}
\sum_{1 \leq k < \ell \leq m\atop
{\bf c}_k + {\bf c}_\ell \,=\, {\bf d}} \!\! P_{k\ell} ({\bf u},{\bf v},{\bf w}) \, a_k a_\ell.
\end{equation}

\begin{corollary} \label{cor:hiro}
The Hirota variety $\mathcal{H}_\mathcal{C}$ is defined by
the quartics $P_{k\ell}$ for all unique pairs $(k,\ell)$
and by the polynomials (\ref{eq:nonunique}) 
for all non-uniquely attained points $ {\bf d} \in \mathcal{C}^{[2]}$.
If all points in $\,\mathcal{C}^{[2]}$ are uniquely attained then
$\mathcal{H}_\mathcal{C}$
  is defined by   the $\binom{m}{2}$ quartics
$P_{k\ell}({\bf u},{\bf v},{\bf w})$, so its equations do not involve
the coefficients $a_1,\ldots,a_m$.
This is the case when
 $\,\mathcal{C}$ is the vertex set of a simplex.
\end{corollary}

\begin{example}[The Square]
Let $g=2$ and $\mathcal{C} = \{0,1\}^2$ as in Example~\ref{ex:ClaudiaYelena}.
Here, $\mathcal{C}^{[2]} = \{(0,1),(1,0),(1,1),(1,2),(2,1)\}$.
The Hirota variety $\mathcal{H}_\mathcal{C}$ is a complete intersection of
codimension three in $(\mathbb{K}^*)^4 \times \mathbb{WP}^5$.  
 There are four unique pairs $(k,\ell)$ and these contribute the two quartics
 $P_{13} = P_{24} =  u_1^4 - 4 u_1 w_1 + 3 v_1^2  $ and
 $P_{12} = P_{34} =  u_2^4 - 4 u_2 w_2 + 3 v_2^2  $.
 The point ${\bf d} = (1,1)$  is not uniquely attained in $ \mathcal{C}^{[2]}$.
 The polynomial (\ref{eq:nonunique}) contributed by this ${\bf d}$ equals
 \begin{equation}
 \label{eq:PP}
  P(u_1+u_2,v_1+v_2,w_1+w_2) \, a_{00} a_{11}\, \, + \,\,
  P(u_1-u_2,v_1-v_2,w_1-w_2)  \,a_{01} a_{10} .
 \end{equation}
 For any point in $\mathcal{H}_\mathcal{C}$, we can write
   $\tau(x,y,t)$ as a $ (2,4)$-soliton, as 
   shown in \cite[\S~2.5]{Kodamabook}.
\end{example}

\begin{example}[The Cube]
Let $g=3$ and consider the tropical degeneration of a smooth plane quartic
to a rational quartic. By \cite[Example 6]{struwe}, the associated
theta function equals
\begin{equation}
\label{eq:cubetheta} \begin{matrix} \theta_\mathcal{C} \,\, = \,\, 
& a_{000} \,+ \, a_{100} \,{\rm exp}[z_1] \, + \,
 a_{010} \,{\rm exp}[z_2] \,+ \, a_{001} \,{\rm exp}[z_3]\, + \, a_{110} \,{\rm exp}[z_1+z_2]  
 \\  & \,\,\qquad + \,\, a_{101} \,{\rm exp}[z_1+z_3] \, + \,
 a_{011} \,{\rm exp}[z_2+z_3] \, + \, a_{111} \,{\rm exp}[z_1+z_2+z_3] .
 \end{matrix}
 \end{equation}
 We compute the Hirota variety in $(\mathbb{K}^*)^8 \times \mathbb{WP}^8$.
 The set $\mathcal{C}^{[2]}$ consists of $19$ points. Twelve
 are uniquely attained, one for each edge of the cube. These give
 rise to the three familiar quartics $u_j^4 - 4 u_j w_j + 3 v_j^2$,
  one for each edge direction ${\bf c}_k - {\bf c}_\ell$.
 Six points in $\mathcal{C}^{[2]}$ are attained twice. They
 contribute equations like~(\ref{eq:PP}), one for each of the
 six facets of the cube. Finally, the point ${\bf d} = (1,1,1)$ is 
 attained four times. 
The polynomial (\ref{eq:nonunique}) contributed by ${\bf d}$ equals
 \begin{equation}
 \label{eq:cubecenter} \begin{matrix} 
   &   P(\,u_1 + u_2 + u_3, \,v_1 + v_2 + v_3, \,w_1 + w_2 + w_3\,)\, a_{000} a_{111}  \\
+ &   P(\,u_1 + u_2 - u_3, \, v_1 + v_2 - v_3, \,w_1 + w_2 - w_3\, )\, a_{001} a_{110}  \\
+ &   P(\,u_1 - u_2 + u_3, \,v_1 - v_2 + v_3, \,w_1 - w_2 + w_3\,)\, a_{010} a_{101}  \\
+ &   P(\,-u_1 {+} u_2 {+} u_3, \,-v_1 {+} v_2 {+} v_3, \,-w_1 {+} w_2 {+} w_3\,)\, a_{100} a_{011}.  
\end{matrix}
\end{equation}
We now restrict to the $9$-dimensional component of  $\mathcal{H}_\mathcal{C}$
that lies in $\{ a_{000} a_{110} a_{101} a_{011} \,=\,a_{001} a_{010} a_{100} a_{111} \}$.
Its image in $\mathbb{WP}^8$ has dimension $5$, with fibers that are
cones over $\PP^1 {\times} \PP^1 {\times} \PP^1$.
They are defined by seven equations arising from non-unique $(k,\ell)$. Six of these
are binomials~(\ref{eq:PP}). Extending \cite[\S 2.5]{Kodamabook}, we identify $\tau(x,y,t)$
with $(3,6)$-solitons for 
\begin{equation}
\label{eq:dreisechs}
 A \,\,\, = \,\,\, \begin{small} \begin{pmatrix} 
1 & 1 & 0  & 0 & 0 & 0 \\
0 & 0 & 1 & 1 & 0  & 0 \\
0 & 0 & 0 & 0 & 1 & 1  \end{pmatrix} . \end{small} \end{equation}
By definition,  a $(3,6)$-soliton for the matrix $A$ has the form
\begin{equation}\label{eq:solitonTilde}
\tilde \tau(x,y,t) \,\, = \,\,
\sum_I \prod_{i,j \in I \atop i < j} (\kappa_j - \kappa_i) \cdot
{\rm exp}  \biggl[\, x \cdot \sum_{i \in I} \kappa_i  \,  + \,
y \cdot \sum_{i \in I} \kappa_i^2 \,+\,
t \cdot \sum_{i \in I} \kappa_i^3 \, \biggr] ,
\end{equation}
where $I$ runs over the eight bases
$\,135, 136, 145, 146, 235, 236, 245, 246$. To get from (\ref{eq:cubetheta}) to this form,
we use the following parametric representation of the main component in $\mathcal{H}_C$:
$$  \begin{matrix} u_1 = \kappa_1-\kappa_2\,,\quad
 v_1 = \kappa_1^2- \kappa_2^2 \, , \quad
 w_1 = \kappa_1^3-\kappa_2^3 \, , \\
 u_2 = \kappa_3-\kappa_4 \, ,\quad
 v_2 = \kappa_3^2-\kappa_4^2 \, , \quad
 w_2 = \kappa_3^3 - \kappa_4^3\, , \\
 u_3 = \kappa_5- \kappa_6 \, , \quad
 v_3 = \kappa_5^2-\kappa_6^2\, ,\quad
 w_3 = \kappa_5^3-\kappa_6^3 \, , \\
 a_{111} = (\kappa_3-\kappa_5)(\kappa_1-\kappa_5) (\kappa_1-\kappa_3) \lambda_0 \lambda_1 \lambda_2 \lambda_3 \,, \,\,\,\,
 a_{011} = (\kappa_3-\kappa_5)(\kappa_2-\kappa_5)(\kappa_2-\kappa_3) \lambda_0 \lambda_2 \lambda_3\,, \\
 a_{101} = (\kappa_4-\kappa_5)(\kappa_1-\kappa_5)(\kappa_1-\kappa_4) \lambda_0 \lambda_1 \lambda_3\, ,\,\,\,\,
 a_{001} = (\kappa_4-\kappa_5)(\kappa_2-\kappa_5)(\kappa_2-\kappa_4) \lambda_0 \lambda_3\,, \\
 a_{110} = (\kappa_3-\kappa_6)(\kappa_1-\kappa_6)(\kappa_1-\kappa_3) \lambda_0 \lambda_1 \lambda_2 \, , \,\,\,\,
 a_{010} = (\kappa_3-\kappa_6)(\kappa_2-\kappa_6)(\kappa_2-\kappa_3) \lambda_0 \lambda_2 \, ,  \\
 a_{100} = (\kappa_4-\kappa_6)(\kappa_1-\kappa_6)(\kappa_1-\kappa_4) \lambda_0 \lambda_1 \, , \,\,\,\,
 a_{000} = (\kappa_4-\kappa_6)(\kappa_2-\kappa_6)(\kappa_2-\kappa_4) \lambda_0 .
 \end{matrix}
$$
If we multiply (\ref{eq:solitonTilde}) by 
$\, {\rm exp} \bigl[\, - (\kappa_2 + \kappa_4 + \kappa_6)\, x -
(\kappa_2^2 + \kappa_4^2 + \kappa_6^2) \,y - 
(\kappa_2^3 + \kappa_4^3 + \kappa_6^3) \, t \,  \bigr]\,$
then we obtain the desired function
$\,\theta_\mathcal{C} ({\bf u} x + {\bf v} y + {\bf w} t)\,$
for the above generic point on the Hirota variety.
The extraneous exponential factor disappears after
we pass from $\tilde \tau(x,y,t)$ to $\partial_x^2 \, {\rm log}(\tilde \tau(x,y,t))$. Both
 versions of the $(3,6)$-soliton satisfy (\ref{eq:hirotaF}) and they
represent the same solution to the KP equation (\ref{eq:KP1}).
An analogous construction works for the cube $\mathcal{C} = \{0,1\}^g$  in any dimension~$g$.
\end{example}

We now consider the simplex $\mathcal{C} = \{{\bf 0}, {\bf e}_1,\ldots,{\bf e}_g\}$.
This arises from plane quartics ($g=3$) that degenerate to 
 four lines or to a conic plus two lines \cite[Example 5]{struwe}.
The tau function~is
$$
\tau(x,y,t) \, = \,
a_0 \,+\,
a_1 \,{\rm exp}[u_1 x + v_1 y + w_1 t] \, + \,
a_2 \,{\rm exp}[u_2 x + v_2 y + w_2 t] \, + \,\cdots \,+\,
a_g \,{\rm exp}[u_g x + v_g y + w_g t] .
$$
We know from Corollary  \ref{cor:hiro} that the conditions imposed by Hirota's differential equation
(\ref{eq:hirotaF}) do not depend on ${\bf a}$ but only on ${\bf u},{\bf v},{\bf w}$.
 We thus consider the Hirota variety $\mathcal{H}_\mathcal{C}$ in  $\mathbb{WP}^{3g-1}$.

\begin{lemma} \label{lem:para1}
The Hirota variety $\mathcal{H}_\mathcal{C}$ is the union of two irreducible components of dimension $g$
in $\mathbb{WP}^{3g-1}$.
One of the two components has the following parametric representation:
\begin{equation}
\label{eq:para1}
 u_j \,\mapsto\, \kappa_j - \kappa_0 \, , \,\,\,
     v_j \, \mapsto \, \kappa_j^2 - \kappa_0^2 \, , \,\,\,
     w_j \, \mapsto \, \kappa_j^3 - \kappa_0^3 \quad {\rm for} \,\,\, j = 1,2,\ldots,g.
\end{equation}     
The other component is obtained from (\ref{eq:para1}) by the sign change $\,v_j \mapsto -v_j$ for $j=1,\ldots,g$.
\end{lemma}

\begin{proof} By Corollary  \ref{cor:hiro}, the variety $\mathcal{H}_\mathcal{C}$ is defined by 
the quartics $P(u_i,v_i,w_i)$ and $ P(u_i-u_j, v_i-v_j, w_i-w_j)$. The first
$g$ quartics imply 
$u_j = \kappa_j - \kappa_{j+g}$,
$v_j = \kappa_j^2 - \kappa_{j+g}^2$,
$w_j = \kappa_j^3 - \kappa_{j+g}^3$ for $j=1,\ldots,g$.
Under these substitutions, the remaining $\binom{g}{2}$ quartics
factor into products of expressions $\kappa_i-\kappa_j$.
Analyzing all cases up to symmetry reveals the two components.
\end{proof}

Setting $t = \kappa_0$ and $\kappa_j = u_j + t$, the parameterization
 (\ref{eq:para1}) of $\mathcal{H}_\mathcal{C}$ can be written as follows:
\begin{equation}
\label{eq:para2}
 u_j \, \mapsto \, u_j \, , \,\,\,
     v_j \, \mapsto \, 2u_j t + u_j^2 \,, \,\,\,
     w_j \,\mapsto  \,3 u_j t^2 +3 u_j^2 t + u_j^3 \quad {\rm for} \,\,\, j = 1,2,\ldots,g.     
\end{equation}     

\begin{theorem} \label{thm:khovanskii}
The prime ideal of the  Hirota variety in (\ref{eq:para1}) is minimally generated~by
\begin{itemize}
\item[(a)] the $\binom{g}{2}$ cubics   $\,  \,  \underline{v_i u_j}-v_j u_i \,-\, u_i u_j (u_i-u_j)\,$ for $1 \leq i <  j \leq g$,
\item[(b)] the $g$ quartics $  \,     4 \underline{w_i u_i}-3 v_i^2  \,-\, u_i^4\,$ for $i=1,\ldots, g$,
\item[(c)] the $g(g-1)$ quartics 
$   \,   4 \underline{w_j u_i} - 3 v_i v_j \,+\, 3 u_i (u_i-u_j) v_j - u_i u_j^3 $ for $i \not= j$, and
\item[(d)] the $\binom{g}{2}$ quintics
$  \, 4 \underline{w_i v_j} - 4 w_j v_i \, \,+ 3 u_i v_j (v_j-v_i) + u_i v_j (u_j-u_i) (u_i-2 u_j) + u_i u_j^3 (u_i-u_j)$.
\end{itemize}
These $2g^2 - g$  ideal generators are a minimal Gr\"obner basis with the
 underlined leading terms.\end{theorem}

\begin{proof}
Consider the subalgebra of $\mathbb{K}[t,u_1,\ldots,u_g]$
generated by the $3g$ polynomials in
the parametrization (\ref{eq:para2}). We  sort terms 
by $t$-degree. We claim that this is a 
{\em Khovanskii basis}, or {\em canonical basis}, as defined in 
\cite{KM} or \cite[Chapter 11]{gbcp}.
The parametrization given by
the leading monomials 
$\, u_j \mapsto u_j ,\,
     v_j \, \mapsto  2u_j t,\,
          w_j \mapsto  3 u_j t^2 \,$
          defines a toric variety, namely 
the embedding of
$\PP^1 \times \PP^{g-1}$,
into $\PP^{3g-1}$ by the very ample line bundle $\mathcal{O}(2,1)$.
Its toric ideal is generated by the leading binomials
$ v_i u_j-v_j u_i ,\,4 w_i u_i-3 v_i^2,\,
4 w_j u_i - 3 v_i v_j ,\,
w_i v_j - w_j v_i\,$ seen in (a)-(d).
In fact, by \cite[\S 14.A]{gbcp},
these $2g^2-g$ quadrics form a square-free Gr\"obner basis 
with underlined leading monomials. Under the correspondence in
\cite[Theorem 8.3]{gbcp}, this initial ideal corresponds to
a unimodular triangulation of the associated polytope $(2 \Delta_1) \times \Delta_{g-1}$.

One checks directly that the polynomials (a), (b), (c), (d) vanish
for (\ref{eq:para2}).
Since only two indices $i$ and $j$ appear, 
by symmetry, it suffices to do this check for $g=2$.
Hence the generators of the toric ideal are the leading binomials of certain
polynomials that vanish on the Hirota variety. 
By \cite[Theorem 2.17]{KM} or \cite[Corollary 11.5]{gbcp}, this proves
the Khovanskii basis property. Geometrically
speaking, we have constructed a toric degeneration from
the Hirota variety to a toric variety in $\mathbb{WP}^{3g-1}$.
Furthermore, 
using \cite[Proposition 5.2]{KM} or \cite[Corollary 11.6 (1)]{gbcp} we conclude that the polynomials in (a)-(d) are a
Gr\"obner basis for the prime ideal of (\ref{eq:para2}),
where the term order is chosen to select the underlined leading terms.
\end{proof}

Using the methods described above, we can compute the
Hirota variety $\mathcal{H}_\mathcal{C}$ for each of the
known Delaunay polytopes $\mathcal{C}$, starting
with those in Proposition \ref{prop:delaunaypolytopes}.
We did this above for the triangle, the square, the tetrahedron and  the cube.
Here is one more example.

\begin{example}[Triangular prism] \label{ex:triangprism}
Let $g=3$ and take $\theta_\mathcal{C}$ to be the six-term theta function 
$$  a_{000} \, + \, a_{100} \,{\rm exp}[z_1] \,+ a_{001} \,{\rm exp}[z_3]
\, + \, a_{101} \,{\rm exp}[z_1+z_3] \, + \,
 a_{011} \,{\rm exp}[z_2+z_3] \, + \, a_{111} \,{\rm exp}[z_1+z_2+z_3] .
$$
 The  prism $\mathcal{C}$ arises in  the degeneration
  as in Theorem \ref{degthetathm}
  from a smooth quartic to
 a cubic plus a line. This is the second diagram in Figures 1 and 2
 in \cite[page 11]{agostini&co}.
 The Hirota variety is cut out by four quartics in
 $u_i,v_i,w_i$, one for each edge direction, 
 plus three  relations  involving the $a_{ijk}$,
 one for each of the three quadrangle facets. The  edges
 from the two triangle facets define a reducible variety of codimension $3$.
  One  irreducible component is given~by
 $$
 \langle \, u_1^4 + 3v_1^2 - 4u_1 w_1,\,\,
 u_2^4 + 3v_2^2 - 4u_2w_2,\,\,
  u_1^2 u_2 + u_1 u_2^2 - u_2 v_1 + u_1 v_2 \rangle .$$
  Together with the four other relations, this defines
  an irreducible variety of codimension $4$ inside
  $(\mathbb{K}^*)^6 \times \mathbb{WP}^8$. That irreducible Hirota variety has  the  parametric representation
  $$  \begin{matrix} u_1 = \kappa_1-\kappa_2\,,\quad
 v_1 = \kappa_1^2- \kappa_2^2 \, , \quad
 w_1 = \kappa_1^3-\kappa_2^3 \, , \\
 u_2 = \kappa_2-\kappa_3 \, ,\quad
 v_2 = \kappa_2^2-\kappa_3^2 \, , \quad
 w_2 = \kappa_2^3 - \kappa_3^3\, , \\
 u_3 = \kappa_4- \kappa_5 \, , \quad
 v_3 = \kappa_4^2-\kappa_5^2\, ,\quad
 w_3 = \kappa_4^3-\kappa_5^3 \, , \\
a_{000} = (\kappa_1 - \kappa_4) \lambda_0  \, ,\,\quad
a_{100} = (\kappa_2- \kappa_4)\lambda_0 \lambda_1 \, ,\,\quad
a_{110} = (\kappa_3- \kappa_4)\lambda_0 \lambda_1 \lambda_2 \, , \\
a_{001} = (\kappa_1- \kappa_5) \lambda_0 \lambda_3 \, , \,\,\,
a_{101} = (\kappa_2- \kappa_5)\lambda_0 \lambda_1 \lambda_3 \, ,\,\,\,
a_{111} = (\kappa_3- \kappa_5)\lambda_0 \lambda_1 \lambda_2 \lambda_3 \,.
\end{matrix}
$$
This allows us to write the $\tau$-function
 as a $(2,5)$-soliton, with 
$ A \,\,\, = \,\,\, \begin{pmatrix} 
1 & 1 & 1  & 0 & 0 \\
0 & 0 & 0 & 1 & 1 
\end{pmatrix}.
$
The six bases of the matrix $A$ correspond to the six terms in $\theta_\mathcal{C}$,
in analogy to the cube (\ref{eq:cubetheta}).
\end{example}

\section{The Sato Grassmannian}
\label{sec4}

The Sato Grassmannian is a device for encoding all solutions to the KP equation.
Recall that  the classical Grassmannian  ${\rm Gr}(k,n)$ parametrizes
$k$-dimensional subspaces of $\mathbb{K}^n$.  It is a projective variety in
 $\PP^{\binom{n}{k}-1}$, cut out by 
 quadratic relations known as {\em Pl\"ucker relations}.
 Following \cite[Chapter 5]{INLA}, the Pl\"ucker coordinates
 $p_I$ are indexed by $k$-element subsets $I$ of $\{1,2,\ldots,n\}$.
 As is customary in Schubert calculus \cite[\S 5.3]{INLA},
 we identify these $\binom{n}{k}$ subsets  with partitions $\lambda$ that fit into
 a $k \times (n-k)$ rectangle. Such a partition $\lambda$ is a
sequence $(\lambda_1,\lambda_2, \ldots, \lambda_k)$  of integers that satisfy
$n-k \geq \lambda_1 \geq \lambda_2 \geq \cdots \geq \lambda_k \geq 0$.
The corresponding Pl\"ucker coordinate $c_\lambda = p_I$ is the 
maximal minor of a $k \times n$ matrix $M$ of unknowns, as in \cite[\S 5.1]{INLA},
where the columns are indexed by
$I = \{\lambda_k+1, \lambda_{k-1}+2, \ldots, \lambda_2 + k{-}1 , \lambda_1 + k \}$.
With this notation, the Pl\"ucker relations for ${\rm Gr}(k,n)$ 
are quadrics in the unknowns $c_\lambda$.

\begin{example} \label{ex:fivenine}
We revisit \cite[Example 5.9]{INLA} with Pl\"ucker coordinates indexed by partitions.
The Grassmannian ${\rm Gr}(3,6)$ is a $9$-dimensional
subvariety in $\PP^{19}$.
Its prime ideal is generated by $35$ Pl\"ucker quadrics.
These are found easily by the following two lines in {\tt Macaulay2} \cite{M2}:

\smallskip {\tt
R = QQ[c,c1,c11,c111,c2,c21,c211,c22,c221,c222,c3,c31,c311, \\
  \phantom{11}\qquad \qquad    c32,c321,c322,c33,c331,c332,c333];   \quad I = Grassmannian(2,5,R) 
}  

\smallskip

\noindent  The output consists of $30$ three-term relations, like
$  \underline{c_{211} c_{22}} - c_{21} c_{221} + c_2 c_{222} $
and five four-term relations, like
$ \underline{c_{221} c_{31}} - c_{21} c_{321} + c_{11} c_{331} + c \,c_{333}$.
These quadrics form a minimal Gr\"obner basis.
\end{example}

The {\em Sato Grassmannian} SGM is the zero set of the
Pl\"ucker relations in the unknowns $c_\lambda$, where we now
drop the constraint that $\lambda$ fits into a $k \times (n-k)$-rectangle.
Instead, we allow arbitrary partitions $\lambda$. 
What follows is the description of a minimal 
Gr\"obner basis for~SGM.

Partitions are order ideals in the
poset $\mathbb{N}^2$. The set of all order ideals,
ordered by inclusion, is a distributive lattice, known as {\em Young's lattice}.
Consider any two partitions $\lambda$ and $\mu$ that
are incomparable in Young's lattice. They fit into
a common $k  \times (n-k)$-rectangle, for some $k$ and $n$.
There is a canonical Pl\"ucker relation for ${\rm Gr}(k,n)$ that has
leading monomial $c_\lambda c_\mu$. It is known
that these {\em straightening relations} form a minimal Gr\"obner basis
for fixed $k$ and $n$.
This property persists as $k$ and $n-k$ increase, hence yielding a Gr\"obner basis for SGM.

The previous paragraph paraphrases the definition  in \cite{DE, Sato} of
the Sato Grassmannian as an  inverse limit 
of projective varieties. This comes from the diagram of maps
${\rm Gr}(k,n{+}1) \dashrightarrow {\rm Gr}(k,n) $ and
${\rm Gr}(k{+}1,n{+}1) \dashrightarrow {\rm Gr}(k,n)$, where these rational
maps are given by dropping the last index.
This corresponds to {\em deletion} and {\em contraction} in matroid theory
\cite[Chapter 13]{INLA}.
One checks that  the simultaneous inverse limit for $k\rightarrow \infty $ 
and $n-k \rightarrow \infty$ is well-defined.
The straightening relations in our equational description above are those in
  \cite[Example~4.1]{DE}.
  That they form a Gr\"obner basis is best seen using Khovanskii bases
  \cite[Example 7.3]{KM}.

We next present the parametric representation of SGM that is commonly used in KP theory.
Let $V=\mathbb{K}((z))$ be the field of formal Laurent series with coefficients in our ground field $\mathbb{K}$.
Consider the natural projection map $\, \pi\colon V \to \mathbb{K}[z^{-1}] \,$
onto the polynomial ring in $z^{-1}$. We regard $V$ and $ \mathbb{K}[z^{-1}] $ as
 $\mathbb{K}$-vector spaces, with Laurent monomials $z^i$ serving as bases.
Points in the Sato Grassmannian {\rm SGM} correspond to
	  $\mathbb{K}$-subspaces $U\subset V$ such that
\begin{equation}
\label{eq:kernelcokernel}
	\dim \operatorname{Ker} \pi_{|U} \,\,= \,\, \dim \operatorname{Coker } \pi_{|U} ,
\end{equation}	
	and this common dimension is finite.
We can represent $U \in {\rm SGM}$ via a doubly infinite matrix as follows. For any basis $(f_1,f_2,f_3,\dots)$ of $U$,
 the $j$th basis vector is a Laurent series,
\[
 f_j(z) \,\,=\,\, \sum_{i=-\infty}^{+\infty} \xi_{i ,j}z^{i+1}.
\]
Then $U$ is the column span of the infinite matrix $\xi = (\xi_{i,j})$ whose rows are indexed from top to bottom by $\ZZ$ and whose
columns are indexed from right to left by $\NN$. The $i$-th row of $\xi$ corresponds to the coefficients of $z^{i+1}$.
Sato proved that a subspace $U$ of $V$ satisfies (\ref{eq:kernelcokernel}) 
if and only if there is a basis, called a \emph{frame} of $U$,  whose corresponding matrix has the shape
\begin{equation}
\label{eq:xishape}
\xi \,\,=\,\, \begin{small} \begin{pmatrix} \ddots  & \vdots & \vdots & \vdots & \vdots & \cdots & \vdots \\
%  \ddots & 0 & 0 & 0 & 0 & \cdots & 0 \\ 
\cdots & \mathbf{1} & 0 & 0 & 0 & \cdots & 0 \\ 
\cdots & * & \mathbf{1} & 0 & 0 & \cdots & 0 \\ 
\cdots & * & * & \xi_{-\ell,\ell} & \xi_{-\ell,\ell-1} & \cdots & \xi_{-\ell,1} \\ 
\cdots & * & * & \xi_{-\ell+1,\ell} & \xi_{-\ell+1,\ell-1} & \cdots & \xi_{-\ell+1,1} \\ 
{}  & \vdots & \vdots & \vdots & \vdots & \cdots & \vdots \\
\cdots & * & * & \xi_{-1,\ell} & \xi_{-1,\ell-1} & \cdots & \xi_{-1,1} \\
\cdots & * & * & \xi_{0,\ell} & \xi_{0,\ell-1} & \cdots & \xi_{0,1} \\ 
\cdots & * & * & \xi_{1,\ell} & \xi_{1,\ell-1} & \cdots & \xi_{1,1} \\
{}  & \vdots & \vdots & \vdots & \vdots & \cdots & \vdots 
\end{pmatrix}. \end{small}
\end{equation}
This matrix is infinite vertically, infinite on the left and, most importantly, it is eventually lower triangular with $1$ on the diagonal, at the $(-n,n)$ positions. 
The space $U$ is described by the positive integer $\ell$ and the submatrix with
$\ell$ linearly independent columns whose upper left entry is $\xi_{-\ell,\ell}$.
This description implies that a subspace
$U$ of $V$ satisfies (\ref{eq:kernelcokernel}) if and only~if
\begin{equation}
\label{eq:kernelcokernel2}
\text{there exists $\ell \in \mathbb{N}$ such that} \quad
\dim U\cap V_{n} \,\,=\,\, n+1 \quad \text{ for all } n\geq \ell,
\end{equation}
 where $V_n = z^{-n}\mathbb{K}[\![z ]\!]$ denotes the space of Laurent series with a pole of order at most $n$.
 
The Pl\"ucker coordinates on SGM are computed as minors
$\xi_\lambda$ of the matrix $\xi$.
Think of a partition $\lambda$ as a
  weakly decreasing sequence of nonnegative integers
that are eventually zero. 
Setting $m_i = \lambda_i - i$ for $i \in \NN$,
we obtain the  associated {\em Maya diagram}  $(m_1,m_2,m_3,\dots)$. This is a vector of strictly decreasing integers $m_1 > m_2 >  \dots$ such that $m_i = -i$ for large enough $i$.
Partitions and Maya diagrams are in natural bijection.
Given any partition $\lambda$, we consider the
matrix $(\xi_{m_i,j})_{i,j\geq 1}$ whose row indices $m_1,m_2,m_3,\dots$ are the entries
in the Maya diagram of $\lambda$.
 Thanks to the shape of the matrix $\xi$, it makes sense to take the determinant
\begin{equation}
\label{eq:SGMpara}
 \xi_{\lambda} \,\,:= \,\, \det(\xi_{m_i,j}). 
 \end{equation}
This Pl\"ucker coordinate is a scalar in $\mathbb{K}$ that can be
computed as a maximal minor of the finite matrix to the lower right of $\xi_{-\ell,\ell}$
in (\ref{eq:xishape}).
 We summarize our discussion as a theorem.

\begin{theorem} \label{thm:sgm} The Sato Grassmannian SGM is the inverse limit of the
classical Grassmannians ${\rm Gr}(k,n) \subset \PP^{\binom{n}{k}-1}$
as both $k$ and $n-k$ tend to infinity.
A parametrization of SGM is given by the matrix minors 
$c_\lambda = \xi_\lambda$ in (\ref{eq:SGMpara}), where $\lambda$ runs over all partitions.
The equations of SGM are the quadratic Pl\"ucker
relations, shown in  \cite[Example 4.1]{DE} and in Example \ref{ex:fivenine}.
\end{theorem}

We now connect the Grassmannians above to our study of solutions to the KP equation.
Fix positive integers $k < n$ and a vector of parameters
$\kappa = (\kappa_1,\kappa_2,\ldots,\kappa_n)$.
For each $k$-element index set $I \in \binom{[n]}{k}$
we introduce an unknown $p_I$ that serves as a Pl\"ucker coordinate.
Our ansatz for solving (\ref{eq:hirotaF}) is now the following
 linear combination of exponential functions:
\begin{equation}\label{eq:solitonP}
\tau(x,y,t) \,\, = \,\,
\sum_{I \in \binom{[n]}{k}}
p_I \cdot \prod_{i,j \in I \atop i < j} (\kappa_j - \kappa_i) \cdot
{\rm exp} \biggl[\, x \cdot \sum_{i \in I} \kappa_i  \,  + \,
y \cdot \sum_{i \in I} \kappa_i^2 \,+\,
t \cdot \sum_{i \in I} \kappa_i^3 \,\biggr] .
\end{equation}

\begin{proposition} \label{thm:HirotaGrass}
The function $\tau$ is 
 a solution to
Hirota's equation (\ref{eq:hirotaF}) if and only if
the point $p = (p_I)_{I \in \binom{[n]} {k}} $ lies in
the Grassmannian ${\rm Gr}(k,n)$, i.e.~there is
a $k \times n$ matrix $A = (a_{ij})$ such that, for all $I \in \binom{[n]}{k}$,
the coefficient $p_I$ is the $k \times k$-minor
of $A$ with column indices $I$.
\end{proposition}

\begin{proof} This follows from \cite[Theorem 1.3]{Kodamabook}. \end{proof}

We define a  {\em $(k,n)$-soliton} to be any function $\tau(x,y,t)$
where $\kappa \in \RR^n$ and $p \in {\rm Gr}(k,n)$.
Even the case $k=1$ is interesting. Writing
 $A = \begin{pmatrix} a_1 & a_2 & \cdots & a_n \end{pmatrix}$, the $(1,n)$-soliton equals
$$
\tau(x,y,t) \,\, = \,\, \sum_{i=1}^n a_i \,
{\rm exp} \bigl[\, x \cdot  \kappa_i  \,  + \,
y \cdot \kappa_i^2 \,+\,
t \cdot  \kappa_i^3 \,\bigr] .
$$
If we now set $n=g+1$ and we divide the sum above by its first exponential term then we obtain
the tau function that was associated with the $g$-simplex in Lemma~\ref{lem:para1}.
Hence the KP solutions that arise when the Delaunay polytope is a simplex
are precisely the $(1,n)$-solitons.

We next derive the Sato representation in \cite[Definition 1.3]{Kodamabook},
that is, we  express $\tau(x,y,t)$ as a linear combination of Schur polynomials.
Let $\lambda$ be a partition with at most three parts, written
$\lambda_1 \geq \lambda_2 \geq \lambda_3  \geq 0 $.
Following \cite[\S 1.2.2]{Kodamabook},  the associated Schur polynomial
 $\sigma_\lambda(x,y,t)$ can be defined as follows. We first introduce the
	\emph{elementary Schur polynomial} $\varphi_j(x,y,t)$  by the series 
$ {\rm exp} [\, x \lambda + y \lambda^2 + t \lambda^3] = \sum_{j=0}^{\infty} \varphi_j(x,y,t)\lambda^j $.
The \emph{Schur polynomial} $\sigma_\lambda $ for the partition $\lambda = (\lambda_1,\lambda_2,\lambda_3)$ 
is the determinant of the Jacobi-Trudi matrix of size $ 3 \times 3$:
$$ \sigma_{\lambda}(x,y,t) \,\,=\,\, \det \bigl(\varphi_{\lambda_i-i+j}(x,y,t) \bigr)_{1\leq i,j \leq 3} . $$
 To be completely explicit, we list 
Schur polynomials for partitions $\lambda$ with $\lambda_1 + \lambda_2 + \lambda_3 \leq 4$:
$$ \begin{matrix}
\sigma_\emptyset = 1, \quad \sigma_1 = x, \quad
\sigma_{11} = \frac{1}{2}x^2-y,\quad
\sigma_2 = \frac{1}{2}x^2+y,\, 
\sigma_{111} = \frac{1}{6}x^3-xy+t, \quad
\sigma_3 = \frac{1}{6} x^3+x y+t,
\smallskip \\
\sigma_{21} = \frac{1}{3}x^3-t, 
 \, 
\sigma_{211} = \frac{1}{8} x^4-\frac{1}{2}x^2 y-\frac{1}{2}y^2,\,
\sigma_{22} = \frac{1}{12} x^4-t x+y^2, \, 
\sigma_{31} = \frac{1}{8}x^4+\frac{1}{2}x^2y-\frac{1}{2}y^2,\, \ldots
% \sigma_4 = \frac{1}{24}x^4+ \frac{1}{2}x^2 y+tx+ \frac{1}{2}y^2,\, \smallskip 
\end{matrix}
$$
For a partition $\lambda$ as above, we set $\lambda_4 = \cdots = \lambda_k = 0$.
For $I = \{ i_1 < i_2 < \cdots < i_k \}$ we set
$$ \Delta_\lambda( \kappa_i, i \in I) \quad := \quad
{\rm det} \begin{small} \begin{pmatrix}
\kappa_{i_1}^{\lambda_{1}+k-1} & \kappa_{i_2}^{\lambda_{1}+k-1} & \cdots & \kappa_{i_k}^{\lambda_{1}+k-1} 
\smallskip \\
\kappa_{i_1}^{\lambda_{2}+k-2} & \kappa_{i_2}^{\lambda_{2}+k-2} & \cdots & \kappa_{i_k}^{\lambda_{2}+k-2} \\
\vdots & \vdots & \ddots & \vdots \smallskip \\
\kappa_{i_1}^{\lambda_{k}} & \kappa_{i_2}^{\lambda_{k}} & \cdots & \kappa_{i_k}^{\lambda_{k}} 
\end{pmatrix}. \end{small}
$$
The empty partition gives the Vandermonde determinant
$\Delta_{\emptyset}(\kappa_i , i\in I) = \prod_{i,j \in I \atop i < j} (\kappa_j - \kappa_i).$

\begin{lemma} 
\label{lem:hall} The exponential function indexed by $I$ in the formula
(\ref{eq:solitonP}) has the expansion
$$ 
{\rm exp} \biggl[\, x \cdot \sum_{i \in I} \kappa_i  \,  + \,
y \cdot \sum_{i \in I} \kappa_i^2 \,+\,
t \cdot \sum_{i \in I} \kappa_i^3 \,\biggr] \,\, = \,\, \Delta_\emptyset(\kappa_i,i \in I)^{-1} \cdot \!\!\!\!\!\!
\sum_{\lambda_1 \geq \lambda_2 \geq \lambda_3 \geq 0} \!\! \!\! \!\! \Delta_\lambda (\kappa_i, i \in I) \cdot \sigma_\lambda(x,y,t).
$$
\end{lemma}
 
\begin{proof} 
The unknowns
 $x,y,t$ play the role of power sum symmetric functions in $r_1,r_2,\ldots$:
	\[ x \,=\, r_1+r_2+r_3 = p_1(r), \,\,\, y \,=\, \frac{1}{2}(r_1^2+r_2^2+r_3^2) = \frac{1}{2}p_2(r), 
	\,\,\, t \,=\, \frac{1}{3}(r_1^3+r_2^3+r_3^3) \,=\, \frac{1}{3}p_3(r). \]
	It suffices to prove the statement after this substitution.
By \cite[Remark 1.5]{Kodamabook}, we have $\sigma_{\lambda}(x,y,t) = s_{\lambda}(r_1,r_2,r_3)$, where $s_{\lambda}$ is the usual Schur function as a symmetric polynomial, which satisfies
 $\Delta_\lambda(\kappa_i, i\in I) = s_{\lambda}(\kappa_i, i\in I)\cdot \Delta_{\emptyset}(\kappa_i,i\in I)$. 
 Our identity can be rewritten as
	\[ \exp\left[\!\, p_1(w)\cdot p_1(\kappa) + \frac{1}{2} p_2(w)\cdot p_2(\kappa) + \frac{1}{3} p_3(w)\cdot p_3(\kappa) 
	\! \right] \,\,= \!\!\! \sum_{\lambda_1 \geq \lambda_2 \geq \lambda_3 \geq 0} \!\!\! \!
	s_{\lambda}(\kappa_i, i  \in I)\cdot s_{\lambda}(r_1,r_2,r_3) .\]
	This is precisely the classical Cauchy identity, as stated  in \cite[page 386]{Stanley}. 
\end{proof}

By substituting the formula in Lemma \ref{lem:hall} into the right hand side of (\ref{eq:solitonP}),
we obtain :

\begin{proposition}
The $(k,n)$-soliton has the following expansion into Schur polynomials
\begin{equation}
\label{eq:solitonexpansion}
 \qquad \tau(x,y,t) \,\,\,\, = 
\sum_{\lambda_1 \geq \lambda_2 \geq \lambda_3 \geq 0} \!\!\!\! c_\lambda \cdot \sigma_\lambda(x,y,t)\, ,\qquad
{\rm where} \quad c_\lambda = 
\sum_{I \in \binom{[n]}{k}} \!
p_I \cdot \Delta_\lambda(\kappa_i, i \in I).
\end{equation}
\end{proposition}

For any point  $\xi$ in the Sato Grassmannian {\rm SGM} we now define a tau function as follows:
\begin{equation}
\label{eq:satotau}
 \tau_\xi(x,y,t) \,\,=\,\, \sum_{\lambda} \xi_{\lambda}\,\sigma_{\lambda}(x,y,t) .
 \end{equation}
 The sum is over all possible partitions. We can now state the main result of   Sato's theory.
 
  \begin{theorem}[Sato] \label{thm:sato} For any $\xi \in {\rm SGM}$,
  the tau function   $\tau_\xi$ satisfies Hirota's equation  (\ref{eq:hirotaF}).
\end{theorem}

Actually, Sato's theorem is much more general.
From a frame $\xi$ as in (\ref{eq:xishape}), we can define
	\[ \tau(t_1,t_2,t_3,t_4,\dots)\,\, = \,\,\sum_{\lambda} \xi_{\lambda} \,
	\sigma_{\lambda}(t_1,t_2,t_3,t_4,\dots) \]
	The sum is over all partitions.
This  function in infinitely many variables is a solution to the \emph{KP hierarchy}, which is an infinite set of differential equations which generalize the KP equation. Moreover, every solution to the KP hierarchy arises from the Sato Grassmannian in this way. The tau functions that we consider here arise from the general case by setting 
	\[ t_1 = x, \quad t_2 =y, \quad t_3=t, \quad t_4=t_5=\dots = 0. \]
	We refer to \cite[Theorem 1.3]{Kodamabook} for a first introduction and numerous references.
We may also start with an ansatz
 $\,\tau(x,y,t) \,\,=\,\, \sum_{\lambda} c_{\lambda}\,\sigma_{\lambda}(x,y,t)$,
 and examine the quadratic equations in the unknowns $c_\lambda$ that are
 imposed by  (\ref{eq:hirotaF}). This leads to polynomials
 that  vanish on  {\rm SGM}.
 
 \begin{remark}\label{rmk:grassmaninsato}
 We can view  Proposition~\ref{thm:HirotaGrass}
 as a special case of Theorem \ref{thm:sato}, given that the
 Sato Grassmannian contains all
  classical Grassmannians ${\rm Gr}(k,n)$.
 Here is an explicit description.
 We fix distinct scalars  $\kappa_1,\ldots,\kappa_n$ in $\mathbb{K}^*$.
Points in ${\rm Gr}(k,n)$ are represented by matrices $A$ in $\mathbb{K}^{k \times n}$.
Following \cite[\S 3.1]{KX} and \cite[\S 2.2]{Nak}, we turn $A$ into an infinite matrix $\xi$ as in (\ref{eq:xishape}).
Let $\Lambda(\kappa)$ denote the $ \infty \times n$ matrix whose rows are
$ (\kappa_1^\ell, \kappa_2^\ell, \ldots, \kappa_n^\ell)$ for $\ell = 0,1,2,\ldots$.
We define $A(\kappa) := \Lambda(\kappa) \cdot A^T$.
This is the $\infty \times k$ matrix whose $j$th column is given by the coefficients of 
$$ \sum_{i=1}^n \frac{a_{ji}}{1- \kappa_i z} \,\, = \,\,\,
\sum_{\ell= 0}^\infty\, \sum_{i=1}^n \kappa_i^\ell  \,a_{ji} \cdot  z^\ell .
$$
This verifies \cite[Theorem 3.2]{Nak}. Indeed, the double infinite matrix representing $\tau$ equals
$$ \xi \,=\, \begin{bmatrix}
\,{\bf 1} & {\bf 0} \smallskip \\
\,{\bf 0}  & A(\kappa) \end{bmatrix},
$$
where ${\bf 0}$ and ${\bf 1}$ are  infinite zero and identity matrices.
In particular, the first nonzero row of $A(\kappa)$ is at the row $-k$ of $\xi$. 
The corresponding basis $(f_1,f_2,f_3,\dots)$ of the space $U$ is given~by
\begin{equation}\label{eq:basissoliton}  
f_{j} = \frac{1}{z^{k-1}} \sum_{i=1}^n \frac{a_{ji}}{1-\kappa_i z}, \quad \text{ for } j=1,\dots,k, \qquad f_j=\frac{1}{z^{j-1}}, \quad \text{ for } j\geq k+1.
\end{equation}
The Pl\"ucker coordinates $c_\lambda$ indexed by partitions with at most three parts
are certain minors of $A(\kappa)$, and these are expressed
in terms of maximal minors of $A$ by the formula in (\ref{eq:solitonexpansion}).
\end{remark}

\section{Tau Functions from Algebraic Curves}
\label{sec5}

Let $X$ be a smooth projective curve of genus $g$ defined over
a field $\mathbb{K}$ of characteristic zero. In this section we
show how certain Riemann-Roch spaces on $X$ define points
in the Sato Grassmannian {\rm SGM}. Using Theorem \ref{thm:sato},
we obtain KP solutions by choosing appropriate bases of these spaces.
 The relevant theory is known since the 1980s; see \cite{Kodamabook, KX, Sato}. We begin with the exposition in
\cite[\S 4]{Nak}. Our aim is to develop tools to carry this out in~practice.

Fix a divisor $D$ of degree $g-1$ on $X$ and
 a distinguished point $p \in X$, both defined over $\mathbb{K}$.
For any integer $n \in \NN$, we
consider the Riemann-Roch space $ H^0(X,D+np)$.
For $m < n$ there is an inclusion  $ H^0(X,D+mp) \subseteq H^0(X,D+np) $.
As $n$ increases, we obtain a space $H^0(X, D+\infty p)$ of
rational functions on the curve $X$ whose pole order at $p$ is unconstrained.

Let $z$ denote a local coordinate on $X$ at $p$. Each element in
$H^0(X, D+\infty p)$ has a~unique Laurent series expansion in $z$
and hence determines an element in $V = \mathbb{K} (( z))$. 
 Let $m=\operatorname{ord}_p(D)$ be the multiplicity of $p$  in $D$. Multiplication by $z^{m+1}$
defines  the $\mathbb{K}$-linear map:
\[ \iota\,\colon\,  H^0(X, D+\infty p) \,\rightarrow \,V, \qquad
s\,=\,\sum_{n\in \mathbb{Z}} s_nz^n\, \mapsto \, \sum_{n\in \mathbb{Z}}s_nz^{n+m+1}. \] 

\begin{proposition}[\text{\cite[Theorem 4.1]{Nak}}]
\label{prop:smoothcurves}
The space $U  =  \iota(H^0(X,D + \infty p)) \subset V$ 
lies in~SGM.
\end{proposition}

\begin{proof}
The map $\iota$ is injective because a rational function on an irreducible curve $X$ is uniquely determined by its Laurent series.
Setting $V_n = z^{-n}\mathbb{K}[\![z ]\!] \subset V$ as in Section \ref{sec4}, we have
\begin{equation}
\label{eq:RR}
 \dim U\cap V_n  \,=\, h^0(X,D+(n+1)p) \,=\, n+1 \,+ \,h^1(X,D+(n+1)p).
\end{equation}
The second equality is the  Riemann-Roch Theorem, with ${\rm deg}(D) = g-1$. 
Hence (\ref{eq:kernelcokernel2})  holds provided
$ h^1(X,D+(n+1)p)=0$. This happens for $n\geq g-1$, by degree considerations.	
\end{proof}

Following \cite{Nak}, we examine the case $g=2$. 
A smooth curve of genus two is hyperelliptic:
\[ X \,=\, \left\{ y^2 = (x-\lambda_1)(x-\lambda_2)\cdots (x-\lambda_6) \right\}. \]
Here $\lambda_1,\lambda_2,\ldots,\lambda_6 \in \mathbb{K}$ are pairwise distinct.
Let $p$ be one of the two preimages of the point at infinity under the
  double cover $X \to \mathbb{P}^1$.
    Using the local coordinate $z=\frac{1}{x}$ at $p$, we write
\[ y \,=\, \pm\sqrt{(x-\lambda_1)\cdots (x-\lambda_6)} \,\,=\,
\, \pm \frac{1}{z^3}\cdot  \sum_{n=0}^{+\infty} \alpha_n z^n, \]
where $\alpha_0=1$ and the $\alpha_i$ are polynomials in  $\lambda_1,\dots,\lambda_6$.
We consider three kinds of divisors:
$$	D_0 = p \,, \,\,
		D_1 = p_1 \,\,\,\, {\rm and} \,\,\,\,
	D_2 = p_1 + p_2 - p, $$
where $p_1 = (c_1,y_1), \,p_2=(c_2,y_2)$ are general points on $X$.
 For $m\geq 3$, consider the~functions
$$ \begin{matrix}
  g_m(x) &= &\sum_{j=0}^m \alpha_j x^{m-j},&& &&  \\
  f_{m}(x,y) &=& \frac{1}{2}\left( x^{m-3}y + g_m(x) \right), & &  & & \\ 
  h_j(x,y) &= & \frac{f_3(x,y)-f_3(c_j,-y_j)}{x-c_j} & = &  \frac{y+g_3(x)-(-y_j+g_3(c_j))}{2(x-c_j)} \qquad \text{for}
  \,\,j=1,2.
\end{matrix}
$$
These rational functions are series in $z$ with coefficients that are polynomials in $\lambda_1,\dots,\lambda_6$.
We write $U_i $ for the image of the Riemann-Roch space $ H^0(X,D_i + \infty p)$ under the inclusion~$\iota$. 
 
\begin{lemma}[\hbox{\cite[Lemma 5.2]{Nak}}]
\label{lem:nak}
	 The set $\{1,f_3,f_4,f_5,\ldots\}$ is a basis of $U_0$,
the set $\{1,f_3,f_4,f_5,\ldots\} \cup \{ h_1\}$ is a basis of $\,U_1$, and
	  $\{1,f_3,f_4,f_5,\ldots\} \cup \{h_1,h_2\}$ is a basis of $\,U_2$.
\end{lemma}

This lemma furnishes us with an explicit basis
for the $\mathbb{K}$-vector space $U$ in Proposition~\ref{prop:smoothcurves}.
This basis is a frame in the sense of Sato theory.
It gives us the matrix $\xi$ in (\ref{eq:xishape}), from which we compute
 the Pl\"ucker coordinates (\ref{eq:SGMpara}) and 
the tau function (\ref{eq:satotau}). 
This process is a symbolic computation over the ground field
$\mathbb{K}$. No numerics are needed.
For general curves of genus $g \geq 3$, the same is possible,
but it requires  computing  a basis for $U$, e.g.~using \cite{hess}.

Our approach differs greatly from the computation
of KP solutions from the curve $X$ via theta functions as in \cite{agostini&co, dubrovin, Krichever}.
That would require the computation of  the Riemann matrix of $X$,
which cannot be done over $\mathbb{K}$.
This is why we adopted the SGM approach in \cite{KX, Nak}.

We implemented the method  in
\verb|Maple| for $D_0 = p$ on hyperelliptic curves 
over  $\mathbb{K} = \mathbb{Q}(\epsilon)$.
If $\lambda $ is a partition with $n$ parts, then the Pl\"ucker coordinate
$\xi_\lambda$ is the minor given by the $n$ right-most columns 
of $\xi$ and the rows given by the first $n$ parts in the Maya diagram of $\lambda$. 
Since the tau function  (\ref{eq:satotau}) is an infinite sum over all partitions, our code 
does not  provide an exact solution to the Hirota equation (\ref{eq:hirotaF}).
Instead, it computes  the truncated tau function
\begin{equation}
\label{eq:taun}
\tau[ n ] \,\,:= \,\,\sum_{i=1}^n\sum_{\lambda \vdash i} \,\xi_{\lambda}\, \sigma_{\lambda}(x,y,t), 
\end{equation}
where $n$ is the order of precision. In our experiments we evaluated  (\ref{eq:taun})
up to $n=12$ on a range of hyperelliptic curves of genus $g=2,3,4$. 
The first non-zero $\tau[n]$ is $\tau[\,g\,]=\sigma_{(g)}(x,y,t)$ \cite[Proposition 6.3]{Nak}. 
When plugging
(\ref{eq:taun}) into the left hand side of (\ref{eq:hirotaF}), we get an expression 
in $x,y,t$ whose terms of low order vanish.
 The following facts were observed for this expression.
For $n > g+2$, the term of lowest degree has degree $n+g-3$, and the monomial 
that appears in that lowest  degree $n+g-3= 1,2,3,\ldots$ is 
$\,\, x,y,t,xt,yt,t^2,xt^2,yt^2,t^3,xt^3,yt^3,t^4,\dots$.

We use our \verb|Maple|  code to study $(k,n)$-solitons arising from the degenerations in \cite{Nak}. 
Namely, we explore the limit for $\e \to 0$ for hyperelliptic curves of genus $g=n-1$ given by
\begin{equation}\label{eq:goodfamily} 
y^2\,\,=\,\, (x-\kappa_1)(x-\kappa_1-\e)\cdot \cdots \cdot (x-\kappa_n)(x-\kappa_n- \e). 
\end{equation}
Set $ h(z)=(1-\kappa_1 z)\cdots(1-\kappa_n z)$.
For $\e \to 0$ the frame found in Lemma \ref{lem:nak} degenerates to
\begin{equation}\label{eq:degbasis}
U \, = \, \bigl\{ 1 , z^{-n}h(z),z^{-(n+1)}h(z),z^{-(n+2)}h(z),z^{-(n+3)}h(z),\dots\bigr\}.
\end{equation}
Observe that $z^{-n}h(z)$ gets expanded to 
$$ z^{-n} + z^{-(n-1)}\Bigl(- \sum_{i=1}^n \kappa_i\Bigr)+z^{-(n-2)}\Bigl(\sum_{1\leq i \leq j \leq n} \kappa_i \kappa_j\Bigr)+z^{-(n-3)}\Bigl(\sum_{1\leq i \leq j \leq l \leq n} \kappa_i \kappa_j \kappa_l\Bigr) + O(z^{-(n-4)}).$$
Following \cite[\S 7]{Nak}, one multiplies all elements in $U$ by  $h(z)^{-1}$
in order to obtain a soliton solution. 
  By \cite[Theorem 3.2]{Nak}, we obtain a $(1,n)$-soliton solution given by the  matrix 
 \[ A \,\,=\,\, \left( \Bigl(\prod_{i\neq 1} (\kappa_1 - \kappa_i) \Big)^{-1} \quad \Bigl(\prod_{i\neq 2} (\kappa_2 - \kappa_i) \Big)^{-1} \quad \cdots \quad \Bigl(\prod_{i\neq n} (\kappa_n - \kappa_i) \Big)^{-1} \right). \]
  \begin{example}
	The soliton  that arises from the genus $2$ curve given by the polynomial  $f_2(x)$ in (\ref{eq:twocurves}) is a $(1,3)$-soliton given by the matrix
  	$A = \left(\, \frac{1}{2} \quad -1 \quad \frac{1}{2}\,\right) $
  	and parameters $\kappa = ( 1,2,3)$.
  \end{example}

We computed the tau function
for a range of curves over $\mathbb{K} = \mathbb{Q}(\e)$.
Their limit as $\epsilon \rightarrow 0$ is not the same as the 
tau functions obtained from the
combinatorial methods in Section \ref{sec2}:

\begin{example}
	Let $X$ be the hyperelliptic curve of genus $3$ given by $y^2=f(x)$ where $f(x)$ is
	$$(x+1+\epsilon)(x+1+2\epsilon)(x+1+\epsilon+\epsilon^2)(x+1+2\epsilon+\epsilon^2)
	(x+2+\epsilon)(x+2+2\epsilon)(x+2+\epsilon+\epsilon^2)(x+2+2\epsilon+\epsilon^2) .$$
	In Figure \ref{genus3tree} we exihibit the subtree with 8
        leaves that arises from the $8$ roots of $f(x)$ and the
        corresponding metric graph of genus $3$ which maps to it
        under the hyperelliptic covering.

\begin{figure}[ht]
\centering
\begin{tikzpicture}
   \draw[fill] (3,0) circle [radius=0.1];
   \draw[fill] (5,0) circle [radius=0.1];
   \draw[fill] (1.5,.6) circle [radius=0.1];
   \draw[fill] (1.5,-.6) circle [radius=0.1];
   \draw[fill] (6.5,.6) circle [radius=0.1];
   \draw[fill] (6.5,-.6) circle [radius=0.1];
   \draw[ultra thick] (3,0)--(5,0);
   \draw[ultra thick](3,0)--(1.5,.6);
   \draw[ultra thick](3,0)--(1.5,-.6);
   \draw[ultra thick] (5,0)--(6.5,.6);
   \draw[ultra thick] (5,0)--(6.5,-.6);
   \draw[thin] (1.5,.6)--(1, 1);
   \draw[thin] (1.5,.6)--(1, .2);
   \draw[thin] (1.5,-.6)--(1, -.2);
   \draw[thin] (1.5,-.6)--(1, -1);
   \draw[thin] (6.5,.6)--(7, 1);
   \draw[thin] (6.5,.6)--(7, .2);
   \draw[thin] (6.5,-.6)--(7, -.2);
   \draw[thin] (6.5,-.6)--(7, -1);
   
   \node[color=magenta] at (4,0.25) {2};
   \node[color=magenta] at (2,0.65) {1};
   \node[color=magenta] at (2,-.2) {1};
   \node[color=magenta] at (6,0.65) {1};
   \node[color=magenta] at (6,-.2) {1};
   
   \draw[fill] (10,-1) circle [radius=0.1];
   \draw[fill] (10,1) circle [radius=0.1];
   \draw[fill] (14,-1) circle [radius=0.1];
   \draw[fill] (14,1) circle [radius=0.1];
   
   \draw[ultra thick] (10, 1) -- (14, 1);
   \draw[ultra thick] (10, -1) -- (14, -1);
   \draw [ultra thick] plot [smooth, tension=1] coordinates { (10,1) (9.5,0) (10,-1)};
   \draw [ultra thick] plot [smooth, tension=1] coordinates { (10,1) (10.5,0) (10,-1)};
   \draw [ultra thick] plot [smooth, tension=1] coordinates { (14,1) (13.5,0) (14,-1)};
   \draw [ultra thick] plot [smooth, tension=1] coordinates { (14,1) (14.5,0) (14,-1)};
   
   \node[color=magenta] at (13.3,0) {1};
   \node[color=magenta] at (14.7,0) {1};
   \node[color=magenta] at (9.3,0) {1};
   \node[color=magenta] at (10.7,0) {1};
   \node[color=magenta] at (12,1.35) {$\frac{1}{2}$};
   \node[color=magenta] at (12,-.65) {$\frac{1}{2}$};
\end{tikzpicture}
    	\caption{The metric tree (left) and the metric graph (right) for the 
    		curve $X$}\label{genus3tree}
\end{figure}
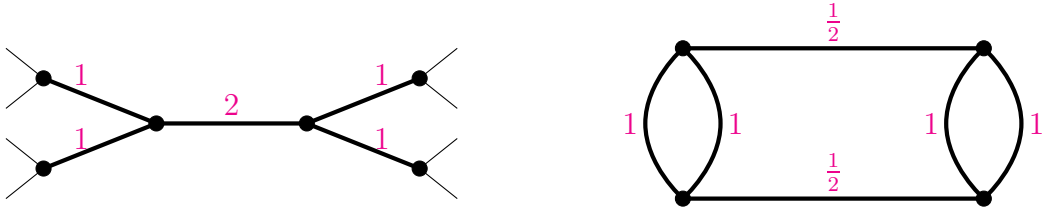

For a suitable cycle basis, the tropical Riemann matrix equals
	$Q = \begin{tiny}
    \begin{bmatrix}
    \, 2 & -1 & 0 \, \\
    \, -1 & 3 & -1 \, \\
    \, 0 & -1 & 2\,\\
    \end{bmatrix}\end{tiny}$.
    This appears in the second row in \cite[Table 4]{struwe}: the Voronoi polytope is
    the hexarhombic dodecahedron. This corresponds to the tropical degeneration from a smooth quartic to a conic and two lines in $\PP^2$. According to \cite[Theorem 4]{struwe} there are two types of  Delaunay polytopes in this case, namely the tetrahedron (4 vertices) and the pyramid (5 vertices). The theta function 
(\ref{eq:thetafunction}) for the tetrahedron equals
     $\,\theta_\mathcal{C}(z) \,\, = \,\,a_{000}+a_{100}\exp [z_1]+a_{010}\exp [z_2]+a_{001}\exp [z_3]$.
The Hirota variety lives in  $(\mathbb{K}^*)^4 \times \mathbb{WP}^8$,
and it is characterized by Theorem \ref{thm:khovanskii}. Each point on the Hirota variety
gives a KP solution. The theta function for  the pyramid equals
     $$\theta_{\mathcal{C}'}(z) \,\,=\,\, a_{000}+a_{100}\exp [z_1]+a_{001}\exp [z_3]+a_{101}\exp [z_1+z_3]+a_{111}\exp [z_1+z_2+z_3].$$
      The Hirota variety $\mathcal{H}_{\mathcal{C}'} \subset
      (\mathbb{K}^*)^5 \times \mathbb{WP}^8$ is cut out
      by eight quadrics $P_{ij}$ as in Section \ref{sec3}, plus
            $$
      P(u_1+u_3,v_1+v_3,w_1+w_3) \, a_{000} a_{101}\, \, + \,\,
      P(u_1,v_1,w_1)  \,a_{001} a_{101}  .
      $$
The resulting tau functions differ from those obtained by setting
$\e = 0$ in our \verb|Maple|  output. This happens because
$y^2=f(x)$ is not a semistable model. The special fiber of that curve at
 $\epsilon = 0$ does not have ordinary singularities: it has two singular points of the form $y^2=x^4$. On the other hand, if the curve at $\epsilon = 0$ is rational and has nodal singularities, as in \eqref{eq:goodfamily}, we do get soliton solutions at the limit. We shall see this more precisely in the next section.
\end{example}

After this combinatorial interlude, we now return to  Proposition  \ref{prop:smoothcurves},
and explore this for a singular curve $X$.
Suppose $X$ is connected, has arithmetic genus $g$, and all singularities are nodal.
We recall briefly how to compute $H^0(X,E)$ when $E$ is a divisor supported in
the smooth locus of $X$. If $X$ is irreducible, then we 
consider the normalization $\widetilde{X} \to X$, that separates the nodes of $X$. 
The divisor $E$ lifts to $\widetilde{X}$, and  $H^0(X,E)$ is a subspace of
$H^0(\widetilde{X},E)$. It consists of rational functions 
 which coincide on the points of $\widetilde{X}$ that map to the nodes of $X$. If $X=X_0\cup \dots \cup X_r$ is reducible,  then $H^0(X,E)$ is a subspace of $\bigoplus_{i=0}^r H^0(X_i,E_{|X_i})$. Its elements are tuples
  $(f_0,f_1,\dots,f_r)$  where $f_i$ and $f_j$ coincide on $X_i\cap X_j$. 

Fix a divisor $D$ of degree $g-1$ and a point $p$, where
all points are smooth on $X$ and defined over $\mathbb{K}$.
We wish to compute $H^0(X,D+\infty p)$.
Riemann-Roch holds for $X$ and hence so does 
(\ref{eq:RR}). In order for the
proof of Proposition \ref{prop:smoothcurves} to go through, we need two conditions:
\begin{enumerate}
 \item[($*$)] A rational function in $H^0(X,D{+}np)$ is uniquely determined by its Laurent series at $p$. 
 \item[($**$)] We have $\,h^1(X,D+np) = 0\,$ for $n\gg 0$.
\end{enumerate}
Our next result characterizes when these two conditions hold.
Let $X_0$ be the irreducible component of $X$ that contains $p$,  and let
$X_0' = \overline{X\backslash X_0}$ be the curve obtained removing $X_0$. 
Set $Z=X_0\cap X_0'$ and denote the restrictions of the divisor $D$ to $X_0,X_0'$ by $D_0,D_0'$ respectively.

\begin{proposition}
\label{prop:cond1and2}
Condition {\rm ($*$)} holds if and only if $H^0(X_0',D_0'-Z)=0$.
Condition {\rm ($**$)} holds if and only if $H^1(X_0',D_0')=0$.
These are vanishing conditions on the curve~$X_0'$.
\end{proposition}

\begin{proof}
If $X$ is irreducible then ($*$) holds since rational functions are determined
	by their series on the normalization $\widetilde{X} $.
If $X$ is reducible then $X_0'$ is nonempty.
	 We need that the~restriction
$$	H^0(X,D+np) \,\longrightarrow \,H^0(X_0,D_0+np) $$
	is injective. The kernel of this map consists exactly of those rational functions in $H^0(X_0',D_0')$ which vanish on $Z$. In other words, the kernel is the space $H^0(X_0',D_0'-Z)$, as desired.

Also for the second statement, we can take $X$ to be reducible.
Consider the exact sequence
	$$ 0 \,\longrightarrow \,\mathcal{O}_{X}(D+np)
	\, \longrightarrow\, \mathcal{O}_{X_0}(D_0+np) \oplus \mathcal{O}_{X_0'}(D_0') 
	\,\longrightarrow \,\mathcal{O}_Z(D_0+np) \, \longrightarrow\, 0. $$
	Taking global sections we see that $H^1(X,D+np)$ surjects onto $H^1(X_0',D_0')$, since
	${\rm dim}(Z) = 0$.
Hence $H^1(Z,D_0+np)=0$. In particular, if $H^1(X,D+np)= 0$ then $H^1(X_0',D_0') = 0$.
	
Conversely, suppose $H^1(X_0',D'_0)=0$. Since $X_0$ is irreducible, we have $H^1(X_0,D_0+np) = 0$ for $n\gg 0$.
	 The long exact sequence tells us that $H^1(X,D+np)=0$ as soon as the map
	\[ H^0(X_0,D_0+np) \oplus H^0(X_0',D_0') \longrightarrow H^0(Z,D+np) ,\qquad (f_0,f_0') \mapsto {f_0}_{|Z} - {f'_0}_{|Z} \]
	is surjective.  Actually, the map $H^0(X_0,D+np) \to H^0(Z,D+np)$ is surjective for $n\gg 0$. 
Indeed,
	$H^1(X_0,D-Z+np) = 0$ for $n\gg 0$ since  $p$ is an ample divisor on the  curve $X_0$.
\end{proof}

\begin{remark}
  Here we presented the case of nodal curves for simplicity, but the same discussion holds true, with essentially the same proofs, for an arbitrary singular curve.
\end{remark}

\begin{remark} The two conditions in 
Proposition \ref{prop:cond1and2} are automatically satisfied when the curve $X$ is irreducible. 
In that case we always get a point $U$ in the Sato Grassmannian.
\end{remark}

\section{Nodal Rational Curves}
\label{sec6}

Our long-term goal is to fully understand the points $U(\e)$
in the Sato Grassmannian that represent Riemann-Roch spaces of
a smooth curve over a valued field, such as $\mathbb{K}= \mathbb{Q}(\e)$.
We explained how these points are computed, and we implemented this in
\verb|Maple| for the case of hyperelliptic curves.
Our approach is similar to \cite{KX, Naktrig, Nak}.
For a given Mumford curve,
it remains a challenge to lift the computation
to the valuation ring (such as $\mathbb{Q}[\e]$) and
 correctly encode the limiting process as  $\e \rightarrow 0$.
 In this section we focus on what happens in the limit.

Consider a nodal reducible curve $X = X_0 \,\cup \,\cdots \,\cup \,X_r$,
where each irreducible component  $X_i$ is rational.
The arithmetic genus $g$ is  the genus of the dual graph.
We present an algorithm whose input is a divisor $D$ of degree $g-1$ and a point $p$,
supported in the smooth locus of $X$. The algorithm checks the
conditions in Proposition \ref{prop:cond1and2}, and, if these
are satisfied, it outputs a soliton solution that corresponds to $U = \iota(H^0(X,D+\infty p))$.

We start with some remarks on interpolation of rational functions on $\mathbb{P}^1$.
 Consider distinct points $\kappa_1,\dots,\kappa_a$ and 
 $\kappa_{1,1},\kappa_{1,2},\dots,\kappa_{b,1},\kappa_{b,2}$ 
 on $\mathbb{P}^1$. We also choose a divisor $D_0=m_1p_1+\dots+m_sp_s+m p$, 
 which is supported away from the previous points. Choose also~scalars 
 $\lambda_1,\dots,\lambda_a,\mu_1,\dots,\mu_b \in \mathbb{K}$. 
 We wish to compute all functions $f$ in $H^0(\mathbb{P}^1,D_0+\infty p)$ satisfying
\begin{equation}\label{eq:interpolation} 
 f(\kappa_j) = \lambda_j \,\text{ for }\, j=1,\dots,a \qquad {\rm and} \qquad f(\kappa_{j,1})=f(\kappa_{j,2})=\mu_j
 \, \text{ for }\, j=1,\dots,b.
\end{equation}

To do so, we choose an affine coordinate $x$ on $\mathbb{P}^1$ such that $p=\infty$. Then we define
\[ P(x)\,:=\, \prod_{j=1}^{s} (x-p_j)^{m_j} \qquad {\rm and} \qquad 
   K(x) \,:= \,\prod_{j=1}^{a}(x-\kappa_j) \cdot \prod_{j=1}^{b}(x-\kappa_{j,1})(x-\kappa_{j,2}) .\]
   Write $K'(x)$ for the derivative of the polynomial $K(x)$. An interpolation argument shows:

\begin{lemma}\label{lemma:interpolation}
A rational function $f$ in $H^0(\mathbb{P}^1,D_0+\infty p)$ satisfies condition \eqref{eq:interpolation} if and only~if 
\[ f (x) \,\,=\,\, \frac{K(x)}{P(x)} \left[ \sum_{j=1}^{a} \lambda_j \frac{P(\kappa_{j})}{K'(\kappa_j)} \frac{1}{x-\kappa_j} + \sum_{j=1}^b \mu_j \left( \frac{P(\kappa_{j,1})}{K'(\kappa_{j,1})}\frac{1}{x-\kappa_{j,1}} +  \frac{P(\kappa_{j,2})}{K'(\kappa_{j,2})}\frac{1}{x-\kappa_{j,2}} \right) + H(x) \right]\!, \]
where $\mu_1,\ldots,\mu_b \in \mathbb{K}$ and $H(x)$ is a polynomial in $\mathbb{K}[x]$.
\end{lemma}

Lemma \ref{lemma:interpolation} gives a way to compute the Riemann-Roch space $H^0(X,E)$ when $E$ is a divisor on a nodal rational curve $X$ as above. The normalization of such a curve is an union of projective lines. On each line we need to compute rational functions with prescribed values at 
certain points (corresponding to the intersection of two components of $X$) and at certain 
pairs of points (corresponding to the nodes in the components of $X$).

%This gives a way to compute the space $H^0(X,E)$ when $E$ is a divisor on an irreducible nodal rational curve $X$. Indeed, if $\nu\colon \mathbb{P}^1 \to X$ is the normalization and $\kappa_{1,1},\kappa_{1,2},\dots,\kappa_{b,1},\kappa_{b,2}$ are the preimages of the nodes, an element in $H^0(X,E)$ is  a rational function $f$ in $H^0(\mathbb{P}^1,E)$ that satisfies the second part of condition \eqref{eq:interpolation}. Then $f$ has the form in Lemma \ref{lemma:interpolation}, and then one just needs to bound the order of the poles at $p$, 
%which is a linear condition in the $\mu_i$.
	
%    We can apply this to compute $H^0(X,E)$ on the reducible curve
%         $X=X_0\cup \dots \cup X_r$ whose components $X_i$
%     are rational and nodal.
%     To do this, we identify the spaces $H^0(X_i,E_{|X_i})$ and then we impose 
%      the requirement that the functions coincide at the intersection points. 

\begin{algorithm_thm} 
\label{alg:sec6}
The following steps compute the soliton 
\eqref{eq:solitonP}
associated to the curve data.

\begin{enumerate}
\item[\textbf{Input:}] A reducible curve $X = X_0 \cup \cdots \cup X_r$ as above,
with a smooth point $p$ and a divisor $D$ of degree $g-1$ supported also on smooth points. 
Everything is defined over~$\mathbb{K}$. 
	
\item[(1)] Let $X_0,X_0',D_0,D_0',Z$ be as in Section \ref{sec5}. Write $Z=\{q_1,\dots,q_a\}$ and 
let $n_1,\dots,n_b$ be the nodes in $X_0$. If $\nu\colon \mathbb{P}^1 \to X_0$ is the normalization of $X_0$ we set $\kappa_j := \nu^{-1}(q_j)$ and $\{ \kappa_{j,1},\kappa_{j,2} \} = \nu^{-1}(n_j)$. We also write $D_0 = m_1p_1+\dots+m_sp_s+mp$, we fix an affine coordinate $x$ on $\mathbb{P}^1$ such that $p=\infty$, and we 
compute $P(x)$ and $K(x)$ as in \eqref{eq:interpolation}.
	
\item[(2)] Compute a basis $Q_1,Q_2,\dots,Q_\ell$ of $H^0(X_0',D_0')$. If $\ell = \deg D_0'+1-p_a(X_0')$ 
then proceed. Otherwise return ``Condition ($**$) in Lemma \ref{prop:cond1and2} fails'' and terminate.
	
\item[(3)] Compute the Riemann-Roch space $H^0(X_0',D_0'-Z)$. If this is zero then proceed.
 Otherwise return ``Condition ($*$) in Lemma \ref{prop:cond1and2} fails'' and terminate.
	
	\item[(4)] Define the $\ell\times a$ matrix $A$ and the $b\times 2b$ matrix $B$ by
	\[ A_{i,j} := \frac{Q_i(p_j)P(\kappa_j)}{K'(\kappa_j)}, \qquad B_{j,2j-1} := \frac{P(\kappa_{j,1})}{K'(\kappa_{j,i})},\,\, B_{j,2j} := \frac{P(\kappa_{j,2})}{K'(\kappa_{j,2})},\,\, B_{i,j} := 0 \text{ otherwise. }   \]
	
	\item[\textbf{Output:}] The $(\ell+b)\times(a+2b)$ matrix $\begin{pmatrix} A & 0 \\ 0 & B \end{pmatrix}$. This represents the soliton solution for  the point $\iota(H^0(X,D+\infty p))$ in the Sato Grassmannian SGM, after a gauge transformation.   
\end{enumerate}
\end{algorithm_thm}

\begin{proposition}\label{prop:algorithm}
Algorithm \ref{alg:sec6} is correct.
\end{proposition}

\begin{proof}
By Riemann-Roch, we have $h^0(X_0',D_0') = h^1(X_0',D_0')+ \deg D_0'+1-p_a(X_0')$. Hence condition ($**$) in Proposition \ref{prop:cond1and2} is satisfied if and only if the condition in step (2) of
Algorithm \ref{alg:sec6} is satisfied. Moreover, condition ($*$) in Proposition \ref{prop:cond1and2} 
is precisely the condition in step (3). Hence, we need to show that the output of the algorithm corresponds to $\iota(H^0(X,D+\infty p))$, after a gauge transformation. However, we know that any element of $H^0(X,D+\infty p)$ can be written as $(f,\sum_{j} \lambda_j Q_j)$ such that $f\in H^0(X_0,D_0+\infty p)$ and
	\[ f(\kappa_j) = \sum_i \lambda_i Q_i(\kappa_j),\text{ for  } j=1,\ldots,a 
	\qquad {\rm and} \qquad f(\kappa_{j,1}) = f(\kappa_{j,2}) \text{ for  } j=1,\ldots,b .\] 
	At this point, Lemma \ref{lemma:interpolation} gives us a basis of $\iota(H^0(X,D+\infty p))$. Remark \ref{rmk:grassmaninsato} shows that this corresponds exactly to the matrix given by the algorithm, after a gauge transformation.
\end{proof}

%We implemented Algorithm \ref{alg:sec6} in \verb|Maple|, and we applied this to the following examples.

We illustrate the algorithm in the following examples.

\begin{example} Let $X$ be an irreducible rational curve with $g$ nodes.
Algorithm \ref{alg:sec6} returns a matrix $B$ for
	a $(g,2g)$-soliton.
	This is consistent with  (\ref{eq:dreisechs}). 
	Note that $X$ is a tropical limit where
	the graph is one node with $g$ loops and the Delaunay polytope is the $g$-cube.
\end{example}

\begin{example}[$g=2$]
Let $X$ be the union of two smooth rational curves $X_0,X_1$ meeting at three points $Z=\{q_1,q_2,q_3\}$. 
This curve is the special fiber of the genus $2$ curve $\{y^2=f_2(x)\}$ in Example~\ref{ex:ClaudiaYelena}.
It corresponds to the graph on the right in Figure \ref{graphs}.
We choose a smooth point $p\in X_0 \backslash Z$, and we consider three different divisors of degree one: $p$, $-2q +3p $ and $3q - 2p$, where $q$ is a smooth point in $X_1$.
We apply Algorithm \ref{alg:sec6} to these three instances.
	
	\begin{itemize}
		\item Take $D = p$. Then $H^0(X_0',D_0') = H^0(\mathbb{P}^1,\mathcal{O}) = \mathbb{K}$ has the constant function $1$ as basis. The conditions in steps (2) and (3) are both satisfied, and
 the algorithm gives us the soliton solution corresponding to the matrix
$\,A = \begin{pmatrix} \frac{1}{K'(\kappa_1)} & \frac{1}{K'(\kappa_2)} & \frac{1}{K'(\kappa_3)} \end{pmatrix} $.
Note that the Delaunay polytope $\mathcal{C}$ is the triangle, and the approach in
Section~\ref{sec3} leads to the
gauge-equivalent matrix	$A = \begin{pmatrix} 1 & 1 & 1 \end{pmatrix} $.
This also arises for $z_1=0$ in Example \ref{ex:triangprism}.
		\item Take $D_2 = -2q+3p$. Then $H^0(X_0',D_0') \cong H^0(\mathbb{P}^1,-2q) = 0$ and the condition in step (2) is not satisfied. Hence we do not get a point in the Sato Grassmannian.
		
		\item Take $D_3 = 3q-2p$. Then $H^0(X_0',D_0') \cong H^0(\mathbb{P}^1,3q)$ has dimension $4$ and the condition in step (2) is satisfied. However $H^0(X_0',D_0'-Z) \cong H^0(\mathbb{P}^1,\mathcal{O}) \ne 0$ so the condition in step (3) is not satisfied, and we do not get a point in the Sato Grassmannian.
	\end{itemize}
\end{example}

\begin{example}[$g=3$]
Consider four general lines $X=X_0\cup X_1\cup X_2\cup X_3$ in $\PP^2$.
Set $X_0\cap X_i = {\kappa_i}$ and $X_i\cap X_j = q_{ij}$ for $i,j \in \{ 1,2,3 \}$. 
We fix the divisor $D=p_1+p_2+p_3-p$, for 
general points $p \in X_0$ and $p_i\in X_i$ for $i=1,2,3$.
After the preperatory set-up in step (1), we
compute  $H^0(X_0',D_0')$ in step (2). This is the space of
functions $(g_1,g_2,g_3)$ in $\bigoplus_{i=1}^3 H^0(X_i,p_i)$ such that
	$g_i(q_{ij}) = g_j(q_{ij})\,\, \text{ for } i,j \in \{ 1,2,3 \}$.
	Choose affine coordinates $x_i$ on $X_i$ for $i=1,2,3$ such that $p_i = \infty$. 
We compute the following basis for  $H^0(X_0',D_0')$:
	\begin{equation*} \begin{matrix}
	Q_1 = \left( 0, \frac{x_2-q_{12}}{q_{23}-q_{12}},  \frac{x_3-q_{13}}{q_{23}-q_{12}} \right), \,\,
	Q_2 = \left(  \frac{x_1-q_{12}}{q_{13}-q_{12}},  0,  \frac{x_3-q_{23}}{q_{13}-q_{23}} \right), \,\,
	Q_3 = \left(  \frac{x_1-q_{13}}{q_{12}-q_{13}},  \frac{x_2-q_{23}}{q_{12}-q_{13}},  0 \right).
	\end{matrix}
	\end{equation*}
Hence $\ell=3$ and the condition in step (2) holds.	We also find that
         $H^0(X_0',D_0'-Z)=0$, so that the condition in step (3) is satisfied as well.
        Algorithm  \ref{alg:sec6}  outputs the soliton matrix
\begin{equation}
\label{eq:threebythree}
	A \,\,\, = \,\,\,
	\begin{pmatrix}
	0 & \frac{\kappa_1-q_{12}}{q_{13}-q_{12}}\frac{1}{K'(\kappa_1)} & \frac{\kappa_1-q_{13}}{q_{12}-q_{13}} \frac{1}{K'(\kappa_1)} \\
	\frac{\kappa_2-q_{12}}{q_{23}-q_{12}} \frac{1}{K'(\kappa_2)} & 0 & \frac{\kappa_2-q_{23}}{q_{12}-q_{13}}\frac{1}{K'(\kappa_2)} \\
	\frac{\kappa_3-q_{13}}{q_{23}-q_{12}} \frac{1}{K'(\kappa_3)} &  \frac{\kappa_3-q_{23}}{q_{13}-q_{23}}\frac{1}{K'(\kappa_3)}  & 0
	\end{pmatrix}.
\end{equation}
The curve $X$ is the last one in \cite[Figure 2]{struwe}.
The Delaunay polytope $\mathcal{C}$ is a tetrahedron, so 
Theorem \ref{thm:khovanskii} applies.
It would be desirable to better understand the relationship between the soliton solution (\ref{eq:threebythree}),
the Hirota variety $\mathcal{H}_\mathcal{C}$, and
the Dubrovin variety in \cite[Example 6.2]{agostini&co}.
	\end{example}

We end with a few words of conclusion. There are two ways to obtain a tau function from a smooth curve: via the theta function and the Dubrovin threefold as in \cite{agostini&co}, or via the Sato Grassmannian as in Section \ref{sec5}. In this paper we presented a parallel for tropical limits of smooth curves: namely, the Delaunay polytopes and Hirota varieties of Sections \ref{sec1}--\ref{sec3}, or the
Sato Grassmannian as in Section \ref{sec6}.
Our next goal is to better understand this process in families. 
An essential step
is to clarify the relation between the degeneration of theta functions
via ${\bf a} \in \mathcal{C}$ as in \eqref{eq:thetaepsilon} and the choice of the divisor $D$ and 
 point $p$ in Algorithm \ref{alg:sec6}.

\noindent
%{\bf Authors' addresses:}
%\smallskip

\noindent Daniele Agostini,  MPI-MiS Leipzig,
\hfill  {\tt daniele.agostini@mis.mpg.de}
\vspace{-4pt}

\noindent Claudia Fevola,  MPI-MiS Leipzig,
\hfill  {\tt claudia.fevola@mis.mpg.de}
\vspace{-4pt}

\noindent Yelena Mandelshtam, UC Berkeley,
\hfill  {\tt yelenam@berkeley.edu}
\vspace{-4pt}

\noindent Bernd Sturmfels,
MPI-MiS Leipzig and UC Berkeley,
\hfill {\tt bernd@mis.mpg.de}
\vspace{-4pt}

\end{footnotesize}


\newpage

\thebibliography{23}
\begin{footnotesize}

\setlength{\itemsep}{-0.6mm}

\bibitem{AG} A.~Abenda and P.~Grinevich:
{\em Rational degenerations of M-curves,
totally positive Grassmannians and KP2-solitons},
Communications in Math.~Physics
{\bf 361} (2018) 1029--1081.

\bibitem{struwe}
D.~Agostini, T.~{\c{C}}elik {\"O}zl{\"u}m, J.~Struwe and B. Sturmfels:
{\em Theta surfaces}, 
Vietnam J.~Math.~(2021).
%Vietnam Journal of Mathematics (2021).

 \bibitem{agostini&co} D.~Agostini, T.~{\c{C}}elik {\"O}zl{\"u}m and B. Sturmfels: {\em The Dubrovin threefold of an algebraic curve}, Nonlinearity (to appear), {\tt arXiv:2005.08244}.
 
 \bibitem{bolognese} B.~Bolognese, M.~Brandt and L.~Chua: {\em From curves to tropical Jacobians and back}, Combinatorial algebraic geometry (2017) pp. 21-45. Springer, New York, NY.
 
 \bibitem{brandt} M.~Brandt:
 {\em Tropical Geometry of Curves},
 PhD dissertation, UC Berkeley, 2020.
 
 \bibitem{CKS}
 L.~Chua, M.~Kummer and B.~Sturmfels: {\em
  Schottky algorithms: classical meets tropical},
    Mathematics of Computation {\bf 88} (2019) 2541--2558.

\bibitem{DE}
J.~Draisma and R.~Eggermont:
{\em Pl\"ucker varieties and higher secants of Sato's Grassmannian},
J.~Reine Angew.~Math. {\bf 737} (2018), 189--215. 

 \bibitem{dubrovin} B.~Dubrovin:~{\em Theta functions and non-linear 
equations},  Russian Math.~Surveys {\bf 36} (1981) 11-92.
 
 \bibitem{dutour} M.~Dutour: {\em The six-dimensional
 Delaunay polytopes},
 European~J.~Comb.~{\bf 25} (2004) 535-548.

\bibitem{erdahl} R.M.~Erdahl and S.S.~Ryshkov:
{\em The empty sphere}, Canadian J.~Math. {\bf 39} (1987) 794--824.

 \bibitem{M2} D.~Grayson and M.~Stillman:
{\em  Macaulay2, a software system for research in algebraic geometry},
available at \url{http://www.math.uiuc.edu/Macaulay2/}.

\bibitem{hess} F.~Hess: {\em Computing Riemann-Roch spaces in algebraic
function fields and related topics},
J.~Symbolic Computation {\bf 33} (2002) 425--445.

\bibitem{jell} P.~Jell: {\em Constructing smooth and fully faithful tropicalizations for Mumford curves},
Selecta Mathematica {\bf 26} (2021), Paper No  60, 23 pp.

\bibitem{KM}
K.~Kaveh and C.~Manon: {\em Khovanskii bases, higher rank valuations, and tropical geometry},
SIAM J.~Appl.~Algebra Geom.~{\bf 3} (2019) 292--336. 
 
\bibitem{Kodamabook}  Y.~Kodama: {\em KP Solitons and the Grassmannians. Combinatorics and    
Geometry and Two-dimensional Wave Patterns},
Briefs in Mathematical Physics {\bf 22}, Springer,  2017.

\bibitem{KX} Y.~Kodama and Y.~Xie: {\em Space curves and solitons of the KP hierarchy:
I.~The $l$-th generalized KdV hierarchy},  
 SIGMA Symmetry Integrability Geom.~Methods Appl.~{\bf 17} (2021), 024, 43 pages.

\bibitem{Krichever}  I. M. Krichever: {\em Methods  of  algebraic  geometry  in  the  theory  of  nonlinear  equations}, Russian Mathematical Surveys {\bf 32} (1977) 185--214.

\bibitem{INLA} M.~Micha{\l}ek and B.~Sturmfels:
{\em Invitation to Nonlinear Algebra},
    Graduate Studies in Mathematics, Vol 211,
 American Mathematical Society, Providence, 2021.
      
\bibitem{Naktrig} A.~Nakayashiki: {\em Degeneration of trigonal curves and solutions of the KP-hierarchy}, Nonlinearity ~{\bf 31} (2018) 3567--3590.
      
 \bibitem{Nak} A.~Nakayashiki: {\em On reducible degeneration of hyperelliptic curves and soliton solutions},
 SIGMA Symmetry Integrability Geom. Methods Appl.~{\bf 15} (2019), Paper No. 009, 18 pp.
  
 \bibitem{Sato} M. Sato: {\em Soliton equations as dynamical systems on infinite dimensional Grassmann manifold}, Nonlinear Partial Differential Equations in Applied Science, North-Holland, 1982, pp. 259--271. 
 
% \bibitem{Segal} G. Segal and G. Wilson: {\em Loop groups and equations of KdV type}, 
% Inst.~Hautes~Etudes Sci.~Publ.~Math. {\bf 61}, (1985) 5--65.
  
 \bibitem{Stanley} R.~Stanley:
 {\em Enumerative Combinatorics: Volume 2}, Cambridge University Press, 2001
 
\bibitem{gbcp} B.~Sturmfels:
{\em Gr\"obner Bases and Convex Polytopes},
American Mathematical Society, University Lectures Series, No 8, Providence, Rhode Island, 1996. 
\bigskip \bigskip  \bigskip
 \end{document}